\documentclass[10pt,reqno]{amsart}

\usepackage[utf8]{inputenc} 

\usepackage[margin=1in]{geometry} 

\usepackage{graphicx} 
\usepackage{float} 

\usepackage{changepage}

 \usepackage[parfill]{parskip} 
 
\usepackage{booktabs} 
\usepackage{array} 
\usepackage{paralist} 
\usepackage{verbatim} 
\usepackage{subfig} 
\usepackage{mathrsfs}
\usepackage{amssymb}
\usepackage{xcolor}
\usepackage{amsthm}
\usepackage{amsmath,amsfonts,amssymb,esint,hyperref}
\usepackage[noabbrev, capitalize]{cleveref}

\usepackage{autonum}

\usepackage{graphics,color}
\usepackage{enumerate, enumitem}
\usepackage{mathtools,centernot}
\usepackage{cases}
\usepackage{amsrefs}
\usepackage{bbm}
\usepackage{xfrac}

\usepackage{bookmark}

\newtheorem{theorem}{Theorem}[section]
\newtheorem{lemma}[theorem]{Lemma}

\newtheorem{corollary}[theorem]{Corollary}
\newtheorem{definition}[theorem]{Definition}
\newtheorem{proposition}[theorem]{Proposition}
\newtheorem{remark}[theorem]{Remark}

\numberwithin{equation}{section} 

\newcommand{\norm}[1]{\left\|#1\right\|}
\newcommand{\abs}[1]{\left|#1\right|}

\newcommand*{\R}{\ensuremath{\mathbb{R}}}

\renewcommand*{\S}{\ensuremath{\mathcal{S}}}

\newcommand*{\N}{\ensuremath{\mathbb{N}}}

\newcommand*{\tr}{\ensuremath{\mathrm{tr\,}}}

\renewcommand{\P}{\mathcal{P}}
\renewcommand{\S}{\mathcal{S}}
\newcommand{\J}{\mathcal{J}}

\newcommand{\quotes}[1]{``#1''}
\renewcommand{\phi}{\varphi}

\renewcommand{\MR}[1]{} 

\usepackage{color, graphicx}
\usepackage{mathrsfs, dsfont}

\usepackage[]{hyperref}
\hypersetup{
    colorlinks=true,       
    linkcolor=red,          
    citecolor=blue,        
    filecolor=red,      
    urlcolor=cyan           
}

\def\div{\mathop{\rm div}\nolimits}    
 
\def\tr{\mathop{\rm Tr}\nolimits}

\newcommand{\be}{\begin{equation}}
\newcommand{\ee}{\end{equation}}
\newcommand{\D}{\mathrm{D}}

\title{Dissipative structures in compressible inviscid fluids}

\author[M. Inversi]{Marco Inversi}
\address[M. Inversi]{Max Planck Institute for Mathematics in the Sciences, Inselstrasse 22, 04103, Leipzig, Germany.}
\email{marco.inversi@mis.mpg.de}

\date{\today}

\subjclass[2020]{35Q35, 76N10, 35Q31, 35D30, 26A45}
\keywords{Isentropic compressible Euler, inhomogeneous incompressible Euler, shocks, energy dissipation}

\begin{document}

\begin{abstract}


This note aims at the following problem. In an ideal density dependent fluid system, is the total energy dissipated on shock type discontinuities? To this end, we study the local energy balance for weak solutions to the isentropic compressible Euler system and derive fine properties of the defect measure. This is done by a careful analysis of the small scale properties of the solutions at the shock discontinuity. By means of the same technique, we also consider the inhomogeneous incompressible case, and, comparing the result, we confirm the general principle that the accumulation of the total energy on codimension one singular structures is a feature that distinguishes compressible and incompressible models. 
\end{abstract}

\maketitle

\section{Introduction}

We consider weak solutions to the isentropic compressible Euler system
\begin{equation} \label{compr Euler} \tag{E} 
\begin{cases}
    \partial_t(\rho u) + \div(\rho u \otimes u) + \nabla p(\rho) = \rho f, 
    \\ \partial_t \rho + \div(\rho u) =0.
\end{cases}
\end{equation}
The latter describes the conservation of momentum and mass in a compressible ideal fluid system under the action of an external force, where the pressure is assumed to be a function of the density. The equations are settled in $(0,T) \times \Omega$, where $\Omega$ is any open set in the Euclidean space. Here, $u$ is the velocity field, $\rho \geq 0$ is the density of the fluid, $f$ is the external force field, and $p: \R \to \R$ is the scalar pressure law. Since we focus on local properties, the initial values and boundary conditions are not specified. We work with weak solutions, which are defined in the usual distributional sense. More precisely, we assume $u \in L^3_{t,x}, \rho \in L^\infty_{t,x}, f \in L^{\sfrac{3}{2}}_{t,x}$\footnote{Any assumption that makes the product $u\cdot f$ well defined would suffice. \label{f cdot u}} and the pressure law $p$ is assumed to be a continuous function such that $p(0)=0$\footnote{Assuming $p(0)=0$ is not restrictive, since $\nabla p(\rho)$ is the quantity involved in the equation. Moreover, with this choice it is readily checked that the internal energy is continuous at $0$.\label{f: continuity of internal energy}}. For such solutions, the local energy balance for \eqref{compr Euler} defines a possibly non-trivial distribution
\begin{equation} \label{eq: defect distribution}
    -\D := \partial_t \left(  \frac{\rho\abs{u}^2}{2} + \P(\rho) \right) + \div \left( u \left( \frac{\rho \abs{u}^2}{2} + p(\rho) + \P(\rho) \right) \right) - \rho f\cdot u,  
\end{equation}
where the internal energy $\P(\rho)$ is given by 
\begin{equation} \label{internal energy}
    \P(\rho) : = \begin{cases} \displaystyle
        \rho \int_1^\rho \frac{p(r)}{r^2} \, dr & \text{ if } \rho >0, 
        \\ 0 & \text{ if } \rho=0. 
    \end{cases}
\end{equation}
This  choice is driven by the equation, as discussed in \cref{S: main results}. For smooth solutions to \eqref{compr Euler}, the distribution $\D$ vanishes. In the periodic setting, this implies that the variation of the total energy is compensated by the work done by the external force. The same result holds with suitable decay properties if $\Omega= \R^d$, or in domains with non-empty boundary, provided that the flow is tangential to the boundary.

The problem of determining a minimal set of regularity assumptions to show that \eqref{eq: defect distribution} is valid with $\D \equiv 0$ is very interesting and has been considered by several authors \cites{FGSW17, ChenYu19, Wu23}. Unlike the homogeneous incompressible case, where the regularity conditions must be imposed only on the velocity field, here multiple choices are available. Regarding the spatial regularity, it is possible to trade smoothness of the velocity field for that of the density. Another possibility is to allow for vacuum, i.e. the density might vanish or be arbitrarily close to $0$ in some regions. In this case, the mathematical features of the system change. To the best of our knowledge, time commutators appear along the computations, and some regularity with respect to the time variable is needed. More precisely, following the classical proofs for the positive part of the celebrated Onsager conjecture \cite{O49} for the incompressible homogeneous Euler equations \cites{Eyink95, CET94, DR00}, the authors in \cites{FGSW17, ChenYu19, Wu23} show that the total energy is conserved for solutions possessing some fractional regularity (Besov or Sobolev). The latter can be imposed in time, space, or both. 

In the present paper, the point of view is rather different. Instead of giving conditions to prove energy conservation, our goal is to study the accumulation of energy on the set where the velocity or the density has a jump discontinuity. The formation of such \quotes{shock-type singularities} is a typical phenomenon in compressible turbulence. The mathematical literature addressing this problem and its variant is extremely rich and will not be reviewed in this paper. See e.g. \cites{BSV22, BDrSV22, BSV23, BSV23bis,  Daf16, Chr07, Chr19} and the references therein. The geometric setting considered here is that of singularities which happen to be organized on codimension one hypersurfaces with Lipschitz regularity. We will show that the defect distribution, often assumed to be a space-time Radon measure, gives mass to such singular structures. 

The simplest compressible model is that of 1D Burgers solutions (in $L^\infty \cap BV$) with a non-empty jump set, which coincides with the region where the energy dissipation is active. See e.g. \cite{Inv25}*{Remark 4.2}. Physical solutions possess nonnegative dissipation, as they arise as vanishing viscosity limits of viscous Burgers \cite{Daf16}. A direct computation shows that for such a toy problem the dissipation is nonnegative if and only if fluid particles are compressed at the shock discontinuity. See \cref{R: burgers shocks}. Moreover, continuous solutions do not dissipate energy \cite{Daf06} and shocks naturally develop for non constant initial data \cite{Alin95}, \cite{Drivas26}*{Theorem 6.1}. 

For incompressible Euler, the situation is expected to change drastically.  The idea of looking at dissipation on surfaces of codimension one dates back to Shvydkoy \cite{Shv09}. The set of solutions considered by the author covers that of classical vortex sheets. A similar problem has been recently addressed in \cites{DRInvN25, InvVi24} for both the homogeneous and density dependent incompressible setting. In all these cases, it turns out that the dissipation measure does not charge codimension one Lipschitz hypersurfaces where the density and the velocity have traces. The key point is to make use of the incompressibility condition to exploit cancellations.  

A possible functional setting to study the accumulation of energy on jump discontinuities organized on codimension one hypersurfaces is that of solutions with bounded variation or deformation. In the incompressible setting, solutions in $L^\infty_{t,x} \cap L^1_t BV_x$ have recently been shown to conserve the kinetic energy \cite{DRINV23}\footnote{It turns out that $L^\infty_{t,x} \cap L^1_t BV_x$ is currently the only critical space for the Onsager conjecture where where the problem of determining energy rigidity or flexibility for solutions to the incompressible Euler equations is solved.}. Following a striking idea by Ambrosio for the renormalization of the transport equation \cite{Ambr04}, the first proof of this fact relies on the optimization of the convolution kernel needed to regularize the equation and making a non-trivial use of the divergence free condition. See also the discussion in \cite{Inv25}, where different proofs based on blow up analysis have been given. Unlike the optimization procedure on the convolution kernel, it turns out that the technique developed in \cite{Inv25} naturally extends to density dependent settings. Together with \cites{DRInvN25, InvVi24}, this is the starting point of our analysis. 

In \cref{T: dissipation on hypersurfaces}, we compute the dissipation on codimension one Lipschitz hypersurfaces in space-time at which both the density and the velocity have one-sided (space-time) traces. In \cref{R: sign condition} we comment on the case of a nondecreasing convex pressure law. Our result is consistent with the fact that some energy is dissipated (not created) on hypersurfaces where fluid particles enter from both sides. This is related to the validity of the Lax entropy conditions, which serve as a criterion to select physical solutions \cites{Daf16,Lax73}. The study relies on the properties of measure divergence fields possessing strong traces, following \cites{DRInvN25, InvVi24} for the incompressible setting.  Then, we assume no vacuum, i.e. there exists $c>0$ such that
\begin{equation} \label{no vacuum}
    \rho(t,x) \geq c >0 \qquad \text{for a.e. } t,x. 
\end{equation} 
In \cref{t: main DR} we establish an approximation formula for the dissipation in terms of regularized versions of the momentum and the density. This is inspired by the analysis of the Duchon--Robert distribution arising in the homogeneous incompressible Euler system. See \cite{Drivas26} for a survey. Then, we focus on solutions with spatial bounded variation or deformation. No time regularity is assumed. In this case, following the method of \cite{Inv25}, we compute the limit of the approximating sequence for $\D$ given by \cref{t: main DR}. We show that a non-trivial dissipation is linked with the jump discontinuity for the normal component of the velocity, which is possible if and only if the divergence of the velocity field has a $(d-1)$-dimensional part. See \cref{T: main bounded var} for a precise statement. Finally, we collect our results in \cref{T: full statement}. 

In the final section, we comment on the inhomogeneous incompressible setting, where similar questions can be posed. Using the same technique, we establish an approximation formula for the defect distribution. Then, in the bounded variation/deformation setting, we show that the dissipation is trivial. See \cref{T: main IIE} and \cref{T: full statement IIE}. This concludes the analysis started in \cite{InvVi24} and is consistent with the fact that dissipation is controlled by (the lower dimensional part of) the divergence of the velocity field. 

\subsection{Plan of the paper}
The paper is organised as follows. \cref{S: main results} is devoted to present our main results. Far from being exhausting, we discuss some connections with the existing literature and briefly present the ideas of the proofs. In \cref{s: tools} we collect some tools. In \cref{s: proofs} we show the proofs of our main results. In \cref{s: IIE} we move to the inhomogeneous incompressible setting. 


\section{Main results} \label{S: main results}

\subsection{Dissipation on codimension one hypersurfaces}

In the case where the density is allowed to vanish, we compute the dissipation measure on space-time Lipschitz hypersurfaces at which both the velocity and the density possess one-sided traces. See \cref{s: tools} for precise definitions. 

\begin{theorem} \label{T: dissipation on hypersurfaces}
Let $(\rho, u) \in L^\infty_{t,x}$ be a weak solution to \eqref{compr Euler} with pressure law $p \in C(\R)$ such that $p(0)=0$ and external force $f \in L^1_{t,x}$. Suppose that the dissipation $\D$ is a space-time Radon measure. Let $\Sigma \subset \Omega \times (0,T)$ be a $\mathcal{H}^{d}$-countably rectifiable set oriented by a unit normal $n = (n_t, n_x) \in \R \times \R^d$ and assume that both $\rho, u$ possess bilateral traces at $\Sigma$, denoted by $\rho^\pm, u^\pm$. Then we have
\begin{equation} \label{eq: dissipation on hypersufaces}
\D \llcorner \Sigma = - (u^+_n - u^-_n) \, \Pi (\rho^+, \rho^-) \, \mathcal{H}^d \llcorner \Sigma, 
\end{equation}
where we set $u_n^\pm : = u^\pm \cdot n_x$ and 
\begin{equation}
    \Pi(\rho^+, \rho^-) : = \begin{cases}
        \displaystyle \frac{p(\rho^+) + p(\rho^-)}{2} - \frac{\rho^+ \rho^-}{\rho^+ - \rho^-} \int_{\rho^-}^{\rho^+} \frac{p(s)}{s^2} \, ds & \text{ if } \rho^+>0, \rho^->0,  \rho^+ \neq \rho^-, 
        \\ \displaystyle \frac{p(\rho^+) + p(\rho^-)}{2} & \text{ if } \rho^+=0 \text{ or } \rho^- =0, 
        \\ 0 & \text{ if } \rho^+ = \rho^-\geq 0, 
    \end{cases}  \label{eq: function Pi}
\end{equation}
\end{theorem}

By the mean value theorem and recalling that $p(0)=0$, the function $\Pi$ is well defined and continuous on $[0, +\infty)^2$. The existence of traces on space-time hypersurfaces is guaranteed if we assume $\rho \in BV_{t,x}, (1,u) \in BD_{t,x}$, for instance. The proof of \cref{T: dissipation on hypersurfaces} relies on the theory of measure divergence fields, following \cites{DRInvN25, InvVi24} for the incompressible Euler equations, both in the homogeneous and density dependent setting. In those cases, it is shown that the dissipation does not give mass to Lipschitz hypersurfaces, coherently with the incompressibility condition. In the case of \eqref{compr Euler}, by \eqref{eq: defect distribution} we have  
$$ -\D = \div_{t,x} \mathcal{V}, \qquad \mathcal{V} : =  \left( \frac{\rho \abs{u}^2}{2} + \mathcal{P}(\rho) , \, u \left( \frac{\rho\abs{u}^2}{2} + p(\rho) + \mathcal{P}(\rho) \right) \right) \in L^\infty_{t,x}.  $$ 
Given a space-time Lipschitz hypersurface $\Sigma$ oriented by a unit normal vector $n \in \mathbb{S}^d$, the concentration of $\D$ in $\Sigma$ is computed in terms of the jump of the normal component of $\mathcal{V}$ at $\Sigma$, that is
$$ - \D \llcorner \Sigma = (\mathcal{V}^+ - \mathcal{V}^-) \cdot n \, \mathcal H^d \llcorner \Sigma. $$
Similarly, we set 
$$\mathcal{U} : = \left( \rho u, \, \rho u\otimes u + p(\rho) \, \textrm{Id} \right), \qquad \mathcal{Z} := (\rho, \, \rho u). $$
By the same reasoning, since the external force is in $L^1_{t,x}$, \eqref{compr Euler} gives the following 
\begin{equation}
    (\mathcal{U}^+ - \mathcal{U}^- ) \cdot n =0, \quad (\mathcal{Z}^+ - \mathcal{Z}^-)\cdot n =0 \qquad \text{ at $\mathcal{H}^d$-a.e. point at $\Sigma$}. \label{eq: no jump of U,Z}
\end{equation}
Then, it is a matter of computing $(\mathcal{V}^+ - \mathcal{V}^-)\cdot n$ knowing \eqref{eq: no jump of U,Z}. 

\begin{remark} \label{R: burgers shocks}
The method outlined above can be used to compute the dissipation associated with 1D Burgers solutions in $L^\infty_{t,x}$ with traces on a space-time hypersurface $\Sigma$ oriented by a unit vector $n= (n_t, n_x) \in \R\times \R$. In this case, $\D$ is defined by 
$$-\D : = \partial_t \frac{u^2}{2} + \partial_x \frac{u^3}{3}. $$
It turns out that 
$$\D \llcorner \Sigma = - \frac{1}{12} \left[ u^+ - u^- \right]^3 n_x \, \mathcal{H}^1 \llcorner \Sigma.  $$
Therefore, the entropy condition $\D \geq 0$ is satisfied if and only if the jump has a negative sign, that is the particles are accumulating around $\Sigma$.  
\end{remark}

\begin{remark} \label{R: sign condition}
By \eqref{eq: dissipation on hypersufaces}, it turns out that the dissipation charges only the portion of $\Sigma$ where $(u^+ - u^-)\cdot n_x \neq 0$. If $\div(u)$ is a Radon measure, this is the portion of $\Sigma$ charged by the divergence of $u$. Assume, in addition, that the pressure law is convex and non-decreasing. In this case, we show in \cref{L: equivalent formulas} that $\Pi(\rho^+, \rho^-) \geq 0$. Thus, given a Lipschitz space-time hypersurface $\Sigma$, the measure $\D \llcorner \Sigma$ is nonnegative if and only if $ (u^+ - u^-)\cdot n_x \leq 0$ in $\mathcal{H}^{d}$-a.e. point in $\Sigma$, which is equivalent to requiring $\div(u) \llcorner \Sigma \leq 0$. This is consistent with the fact that the energy is dissipated in $\Sigma$ if the fluid particles are compressed around the shock discontinuity. 
\end{remark}

\subsection{Assuming no vacuum} In the second part of the paper, we assume the \emph{no vacuum} condition \eqref{no vacuum}. To begin, we derive an approximation formula for the defect distribution associated with weak solutions to \eqref{compr Euler}. We introduce the following notation. Fix a radial convolution kernel $\eta \in C^\infty_c(B_1; [0,1])$  with $\int \eta =1$. For $\ell>0$, we define $\eta_\ell(x) := \ell^{-d} \eta \left(\sfrac{x}{\ell}\right)$. Given $g \in L^1_{\rm loc}(\R^d)$ and $\rho \in L^\infty_{\rm loc} (\R^d)$ satisfying \eqref{no vacuum}, we set
\begin{equation} \label{eq: favre mollification}
    \overline{g} := g* \eta_\ell, \qquad \widetilde{g} := \frac{\overline{\rho g}}{\overline{\rho}} = \frac{(\rho g) * \eta_\ell}{\rho * \eta_\ell}.  
\end{equation}
The latter is known as \emph{Favre mollification} of $g$ \cite{Favre65}\footnote{See also \cite{Aluie13} and the discussion in \cite{SchumZin25}, where the authors compare different regularization schemes for \eqref{compr Euler} and the compressible Navier--Stokes system from the numerical point of view.}. In order to keep the notation as simple as possible, we omit the explicit dependence of $\overline{g}, \widetilde{g}$ on $\ell$ (and $\eta$).

\begin{theorem} \label{t: main DR}
Let $(\rho, u)$ be a weak solution to \eqref{compr Euler}, with $p \in C(\R)$ such that $p(0)=0$ and external force $f$. Assume that $u \in L^3_{t,x}$, $\rho \in L^\infty_{t,x}$ satisfies \eqref{no vacuum} and $f \in L^{\sfrac{3}{2}}$\footref{f cdot u}. The distribution $\D$ defined by \eqref{eq: defect distribution} satisfies 
\begin{equation} \label{formula D}
    \D = \lim_{\ell \to 0} \left( J_1^\ell + J_2^\ell \right), 
\end{equation}
where the limit in \eqref{formula D} is taken in the sense of distributions and we set
\begin{align}
    J_1^\ell & := \overline{\rho} \left(  \widetilde{u} \otimes \widetilde{u} - \widetilde{u \otimes u} \right) : E \widetilde{u}, \qquad \text{with } E\widetilde{u} := \frac{\nabla \widetilde{u} + (\nabla \widetilde{u})^T}{2}, \label{eq: J_1}
    \\ J_2^\ell & := \div(\widetilde{u})\left( p(\overline{\rho}) - \overline{p(\rho)} \right). \label{eq: J_2}
\end{align} 
\end{theorem}

The proof of \cref{t: main DR} relies on careful manipulations on the PDE. The procedure is somewhat similar to the well known Constantin--E--Titi regularization scheme for the incompressible setting \cite{CET94}. To study the evolution of the kinetic energy, it would be natural to multiply the mollified momentum equation by the mollified velocity field $\overline{u}$ and then pass to the limit with respect to the mollification parameter. However, in the compressible setting (as well as in the inhomogeneous incompressible one discussed in \cref{s: IIE}), the system provides equations for the momentum and for the density. In a low regularity setting, it is not clear how to derive an equation for the velocity field. Thus, instead of $\overline{u}$, we multiply the mollified momentum equation by $\widetilde{u}$, which approximates $u$ in the limit. After some computations, we find
\begin{align} \label{eq: heuristic 1}
    \partial_t \left( \frac{\overline{\rho} \abs{\widetilde{u}}^2}{2}\right) + \div \left( \widetilde{u} \left(  \frac{ \overline{\rho} \abs{\widetilde{u}}^2}{2} + \overline{p}  \right) + \overline{\rho u} (\widetilde{u\otimes u} - \widetilde{u} \otimes \widetilde{u} ) \right) = \overline{p} \div (\widetilde{u}) + \overline{\rho f} \cdot \widetilde{u} -  J_1^\ell,
\end{align}
where $J_1^\ell$ is defined by \eqref{eq: J_1}. Before taking the limit $\ell \to 0$, the main issue is to incorporate the term $\overline{p} \div(\widetilde{u})$ into the spatial divergence. This drives the choice of \eqref{internal energy} as the internal energy, and is related to the renormalization theory of the continuity equation (see \cite{Cr09} for a survey). Given $\beta \in C^1(\R;\R)$ and assuming that $\rho, u$ are smooth, we multiply the continuity equation by $\beta'(\rho)$ and by the chain rule we find 
\begin{equation} \label{eq: heuristic 2}
    \partial_t \beta(\rho) + \div(u \beta(\rho)) = \div(u) \left[ \beta(\rho) - \rho \beta'(\rho) \right].  
\end{equation}
Therefore, it would be natural to choose $\beta$ so that when we sum \eqref{eq: heuristic 1} and  \eqref{eq: heuristic 2} the term $p \div(u)$ simplifies with $[\beta(\rho) - \rho \beta'(\rho)] \div(u)$. A direct computation shows that 
\begin{equation}
    \rho \beta'(\rho) - \beta(\rho) = p(\rho) \qquad \Longleftrightarrow \qquad \beta(\rho) = \rho \left( c+ \int_1^\rho \frac{p(r)}{r^2} \, dr\right) \qquad \text{for some } c \geq 0,  
\end{equation} 
Working with weak solutions, we need to mollify the continuity equation as well, and further error terms must be taken into account. This gives rise to the term $J_2^\ell$ defined by \eqref{eq: J_2}.  

Once \cref{t: main DR} is established, we compute the dissipation for bounded solutions where the velocity has spatial bounded deformation and the density has spatial bounded variation. No time regularity is assumed for both the velocity and the density. In the following statement, $\J_{u_t}$ is the jump set in space of the velocity field $u_t$, which is known to be a $(d-1)$-countably rectifiable set oriented by a unit normal $n: \J_{u_t} \to \mathbb{S}^{d-1}$, and $\rho_t^\pm, u_t^\pm$ are respectively the one-sided limits of $\rho_t, u_t$ at $\J_{u_t}$. Consequently, $(u^{\pm}_t)_n : = u^\pm_t \cdot n$ is the normal component of $u^\pm_t$ in the jump set. We refer to \cref{s: tools} for a detailed explanation of the notation. 

\begin{theorem} \label{T: main bounded var}
Let $(\rho,u) \in L^\infty_{t,x}$ be a weak solution to \eqref{compr Euler} with $p \in C(\R)$ such that $p(0)=0$ and external force $f \in L^1_{t,x}$. Assume that $u \in  L^1_t BD_x$ and $\rho \in  L^1_t BV_x$ satisfies \eqref{no vacuum}. Then $\D$ satisfies 
\begin{equation} \label{eq: formula dissipation explicit}
    \D = - \left[ (\Gamma \Psi + \Pi) \,  \div(u_t)\llcorner \J_{u_t} \right] \otimes dt = - \left[ (\Gamma \Psi + \Pi) \, ((u_t^+)_n - (u_t^-)_n) \,  \mathcal{H}^{d-1}\llcorner \J_{u_t} \right] \otimes dt, 
\end{equation} 
where $\Pi(\rho_t^+, \rho_t^-)$ is defined by \eqref{eq: function Pi} and we set 
\begin{equation} \label{eq: function Gamma}
    \Gamma(\rho_t^+, \rho_t^-, u_t^+, u_t^-) : =  \rho_t^+ \rho_t^- \abs{u_t^+-u_t^-}^2 - (p(\rho_t^+) - p(\rho_t^-)) \, (\rho_t^+ -\rho_t^-), 
\end{equation}
\begin{equation}
    \Psi(\rho_t^+, \rho_t^-) :=  \begin{cases}
        \displaystyle \frac{\rho_t^+ + \rho_t^-}{2(\rho_t^+ - \rho_t^-)^2} - \rho_t^+ \rho_t^- \frac{\ln(\rho_t^+) - \ln(\rho_t^-)}{(\rho_t^+ - \rho_t^-)^3} & \text{ if } \rho_t^+ \neq \rho_t^-, 
        \\ \displaystyle (6 \rho_t)^{-1}  & \text{ if } \rho_t^+ = \rho_t^- = : \rho_t. 
    \end{cases} \label{eq: function Psi}
\end{equation} 
\end{theorem}

We show that the dissipation is related to the concentration of the divergence of the velocity field on $(d-1)$-countably rectifiable sets in space, whose time evolution might be very rough. The proof relies on the explicit computation of the limit of terms $J_1^\ell, J_2^\ell$ given by \cref{t: main DR}, following the blow up analysis made by the author in \cite{Inv25} for the inhomogeneous incompressible case. For example, for a $BV$ function, it turns out that $\mathcal{H}^{d-1}$ almost every point is a Lebesgue point or a jump point, and the gradient is absolutely continuous with respect to $\mathcal{H}^{d-1}$. As we recall in \cref{s: tools}, a similar property holds for $BD$ vector fields. 

Some remarks are in order. 
\begin{itemize}
    \item[(i)] The reader might compare our \cref{T: main bounded var} with \cite{FGSW17}*{Theorem 4.3}, where the authors prove energy conservation for weak solutions $(\rho, u)\in L^\infty_{t,x} \cap L^\infty_t BV_x $ to \eqref{compr Euler} such that $\rho_t, u_t$ are continuous in space for almost every time slice\footnote{The pressure law is also assumed to be of class $C^2$.}. In that case, the density is allowed to vanish, and the $\widetilde{u}$ cannot be used as a test function. The major differences from our statement are the continuity in space, which forbids the formation of shocks, and the fact that \eqref{no vacuum} is traded with more time integrability of the spatial derivative, namely $L^\infty_t BV_x$ versus $L^1_t BV_x$. 
    \item[(ii)] For stationary solutions, the notions of trace considered in \cref{T: dissipation on hypersurfaces} and \cref{T: main bounded var} agree. We will show that the function $\Gamma$ defined by \eqref{eq: function Gamma} vanishes and the formulas provided by \cref{T: main bounded var} and \cref{T: dissipation on hypersurfaces} agree. The same holds in the time dependent setting, provided that some conditions on the time disintegration of the time derivatives of the density and the momentum are satisfied. See the discussion in \cref{ss: comparing formulas} for details.  
    \item [(iii)] If the pressure law is convex, it turns out that $\D \geq 0$ if and only if $\div(u) \leq 0$ as space-time measures. See \eqref{eq: formula density} and \eqref{eq: formula dissipation preliminary}. This is consistent with the fact that the dissipation is nonnegative if and only if the particles accumulate at the shock discontinuity. 
\end{itemize}

\subsection{An attempt of general statement}

To summarize our \cref{T: dissipation on hypersurfaces} and \cref{T: main bounded var}, we propose the following statement, which we include for the reader's convenience. See \cref{s: tools} for the notation. 

\begin{corollary} \label{T: full statement} 
Let $(\rho, u) \in L^\infty_{t,x}$ be a weak solution to \eqref{compr Euler} with $p \in C(\R)$ such that $p(0)=0$ and external force $f \in L^1_{t,x}$. Assume that the dissipation $\D$ given by \eqref{eq: defect distribution} is a space-time Radon measure. Suppose that $\Omega \times (0,T) = A \cup B$, where $A, B$ satisfy the following conditions:  
\begin{enumerate}
    \item $A$ is a closed set and $\rho (t,x) =0$ $\mathcal{L}^{d+1}$-almost everywhere in the interior of $A$; 
    \item $\partial A$ is $\mathcal{H}^d$-countably rectifiable and $\rho,u, p$ possess bilateral space-time traces at $\partial A$;  
    \item there are intervals $I_n \subset (0,T)$ and open sets $B_n \subset \Omega$ such that
    $$B = \bigcup_{n} I_n \times B_n, $$
    $\rho \in L^1(I_n; BV(B_n)), u \in L^1(I_n; BD(B_n))$ and there are constants $c_n>0$ such that 
    \begin{equation} \label{eq: local no vacuum}
        \rho(t,x) \geq c_n >0 \qquad \text{a.e. in } I_n \times B_n. 
    \end{equation} 
\end{enumerate} 
Then we have $\abs{\D} (\overline{A}) =0$ and for any $n\in \N$ we have
$$  \D\llcorner B_n = -  (\Gamma \Psi + \Pi) \, ( \mathcal{H}^{d-1} \llcorner (\J_{u_t} \cap B_n)  \otimes (\mathcal{L}^1\llcorner I_n), $$ 
where $\Gamma, \Psi, \Pi$ are given by \eqref{eq: function Gamma}, \eqref{eq: function Psi}, \eqref{eq: function Pi}, respectively. 
\end{corollary}

\section{Tools} \label{s: tools}

We introduce some useful notation that will be kept throughout the manuscript. Given a vector field $u: \R^d \to \R^d$, we denote by $\nabla u, E u $ the distributional gradient and its symmetric part, respectively. Given an open set $\Omega \subset \R^d$, we denote by $\mathcal{M}(\Omega; \R^N)$ the set of Radon measures (i.e. locally finite Borel measures) on $\Omega$ with values in $\R^N$, endowed with the total variation norm $\norm{\mu}_{\mathcal{M}}$ (see e.g. \cite{EG15}). If the domain and the target space are clear from the context, we adopt the shorthand $\mathcal{M}(\Omega; \R^N) = \mathcal{M}_x$. The weak-$*$ convergence of measures is taken in duality with compactly supported continuous test functions. When dealing with distributions, weak-$*$ convergence is taken in duality with $C^\infty_c$. We often denote by $\langle\cdot, \cdot \rangle$ the duality pairing. 

For any $\gamma \in [0,d]$, we denote by $\mathcal{H}^\gamma$ the $\gamma$-dimensional Hausdorff measure and by $\mathcal{H}^d = \mathcal{L}^d$ the Lesbegue measure in $\R^d$. Fix an integer $0\leq m \leq d$. A set $\Sigma \subset \R^d$ is said to be $m$-countably rectifiable if 
$$\Sigma = \bigcup_{i\in \N} \Sigma_i,$$ where $\Sigma_i = f_i(\R^m)$ and $f_i: \R^m \to \R^d$ are Lipschitz maps. A set $\Sigma \subset \R^d$ is said to be $\mathcal{H}^m$-countably rectifiable if $\Sigma = N \cup E $, where $\mathcal{H}^m (N) =0$ and $E$ is $m$-countably rectifiable. See e.g. \cites{AFP00, Fed}.

\subsection{Blow up points} \label{ss:jumps}

We recall the notion of blow up for $L^1$ functions. To fix some terminology, we begin with the scalar valued case and then move to the vector valued one. Given $g \in L^1(\R^d; \R)$, we say that $x_0 \in \R^d$ is a blow up point for $g$ if there exists $v \in L^1(B_1; \R)$ such that
\begin{equation} \label{eq: blow up}
    \lim_{\ell \to 0} \int_{B_1} \abs{g(x_0 + \ell z ) - v(z)} \, dz = 0. 
\end{equation} 
Here, it is crucial to consider $\ell \to 0$, without passing to subsequences. It is clear that whenever $g$ admits a blow up at $x_0$, then the function $v$ is uniquely determined. We denote by $\S_g$ the complement of the set of Lebesgue points for $g$ 
\begin{equation}
    \mathcal{S}_g^c : = \left\{ x_0 \in \R^d \colon \text{ $g$ has a constant blow up at $x_0$} \right\}.  
\end{equation} 
For any $x_0 \in \S_g^c$, we refer to the value $g(x_0)$ as the \emph{approximate limit} of $g$ at $x_0$. By the Lebesgue differentiation theorem, the set $\S_g$ is $\mathcal{L}^d$-negligible. 

We consider another particularly relevant class of blow up points. Let $x_0\in \R^d$ be a blow up point and let $v$ be the corresponding blow up. We say that $x_0$ is a jump point if there exist $g^+(x_0) \neq g^-(x_0) \in \R$ and a direction $n \in \mathbb{S}^{d-1} $ such that 
\begin{equation}
    v = g^+(x_0) \mathds{1}_{B_1^{n_+}} + g^-(x_0) \mathds{1}_{B_1^{n_-}}, \qquad \text{where we set } B_1^{n^\pm} =  \{ y \in B_1 \, : \,  \pm \langle y, n \rangle \geq 0 \}. 
\end{equation}
We denote by $\J_g$ the collection of jump points of $g$. For any $x_0 \in \mathcal{J}_g$, we say that $g^\pm(x_0)$ are the \emph{one-sided Lebesgue limits} of $g$ at $x_0$ with respect to $n$. We will often drop the explicit dependence on $x_0$. The triplet $(g^+, g^-, n)$ is uniquely determined in $\J_g$, up to changing $n$ with $-n$ and $g^+$ with $g^-$. 

Given a scalar function $g \in L^1$, it can be proved that the set of points at which $g$ admits a non-constant blow up is $(d-1)$-countably rectifiable \cite{DelNin21}. We also remark that for a general function $g \in L^1$ without additional properties, there might be blow up points which are not Lebesgue points nor jump points. Since we will focus on Lebesgue points and jump points, it is convenient to denote by 
$$ \mathcal{B}_g := \S_g^c \cup \J_g. $$
We will tacitly assume that the triplet $(g^+, g^-, n_g)$ is defined for $x \in \mathcal{B}_g$, with the convention that $g^+(x) = g^-(x) = g(x) $ and $n(x) = e_d$ for any $x \in  \S_g^c$. Actually, any choice of $n(x)$ for $x\in \S_g^c$ would work. 

In the vector valued case, given $g\in L^1(\R^d; \R^m)$, we define 
$$\mathcal{B}_g := \bigcap_{i=1}^m \mathcal{B}_{g_i}. $$
We claim that the triplet $(g^+, g^-, n)$ is uniquely defined (up to the change of sign of $n$ and the permutation of $g^+$ with $g^-$) at $\mathcal{H}^{d-1}$-a.e. point in $\mathcal{B}_g$. Recall that for any $i=1, \dots, m$ the triplet $(g_i^+, g_i^-, n_{g_i})$ is well defined in $\mathcal{B}_{g_i}$. Hence, the vectors $g^\pm := (g_1^\pm, \cdots, g_m^\pm)$ are well defined in $\mathcal{B}_g$. However, for any $i \neq j$, we have 
$$n_{g_i} (x_0) = n_{g_j}(x_0) \qquad \text{ for $\mathcal{H}^{d-1}$ -a.e. } x_0 \in \J_{g_i}\cap \J_{g_j}. $$
Thus, setting $n(x_0) := n_{g_i} (x_0)$ at $x_0 \in \J_{g_i}$, the resulting vector is normal to $ \bigcup \J_{g_i}$ for $\mathcal{H}^{d-1}$-a.e. $x_0$. Then, we can arbitrarily extend $n$ to $\mathcal{B}_g$, for example, by setting $n(x_0) = e_d $ for any $x_0 \in \mathcal{B}_g \setminus \bigcup \J_{g_i}$. With this notation, for $\mathcal{H}^{d-1}$-a.e. $x_0 \in \mathcal{B}_g$ we have 
\begin{equation} \label{eq: blow up vector}
    \lim_{\ell \to 0} \int_{B_1^{n^\pm_g(x_0)}} \abs{g(x_0 + \ell z) - g^{\pm}(x_0) } \, dz = 0.  
\end{equation} 
We say that $x_0 \in \mathcal{B}_g$ is a Lebesgue point for $g$ if $x_0 \in \S_{g_i}^c$ for any $i = 1, \dots, m$. We say that $x_0 \in \mathcal{B}_g$ is a jump point for $g$ if $g^+(x_0) \neq g^-(x_0)$. When it is clear from the context, we simply denote $n_g$ by $n$. 

\begin{remark} \label{R: blow up of composition}
Take $g\in L^\infty(\R^d; \R^N)$
and a continuous function $\Psi: \R^N \to \R^M$. Let $(g^+, g^-, n_g)$ be the triplet associated with $g$, which is well defined $\mathcal{H}^{d-1}$-a.e. on $\mathcal{B}_g$. It is clear that $(g^+, g^-)$ are bounded Borel functions on $\mathcal{B}_g$. Moreover, since $F$ is uniformly continuous in the (essential) range of $g$, it follows that 
\begin{equation} \label{eq: blow up of composition}
    \lim_{\ell \to 0} \int_{B_1^{n^\pm_g(x_0)}} \abs{\Psi(g(x_0 + \ell z)) - \Psi(g^{\pm}(x_0)) } \, dz = 0 \qquad \text{for $\mathcal{H}^{d-1}$-a.e. $x_0 \in \mathcal{B}_g$}. 
\end{equation}
Therefore, we infer $\mathcal{B}_{g} \subset \mathcal{B}_{\Psi \circ g}$, and 
\begin{equation} \label{eq: triplet of composition}
    ( (\Psi\circ g)^+, (\Psi \circ g)^-, n_{\Psi \circ g} ) = (\Psi(g^+), \Psi(g^-), n_g) \qquad \text{for $\mathcal{H}^{d-1}$-a.e. $x_0 \in \mathcal{B}_g$}. 
\end{equation}
\end{remark}

\subsection{Measure divergence fields} \label{ss: MD fields and traces}

We recall the notion of distributional normal trace \cites{ACM05,Shv09,CCT19,CTZ09}. Given an open set $\Omega \subset \R^d$, we denote by $\mathcal{MD}^\infty(\Omega)$ the set of vector fields $ V \in L^\infty(\Omega;\R^d)$ such that $\div V \in \mathcal{M}(\Omega)$. For $V\in \mathcal{MD}^\infty(\Omega)$, the distributional outer normal trace on $\partial \Omega$ is the distribution supported on $\partial \Omega$ defined by the following Gauss--Green formula
\begin{equation} \label{eq: normal distributional trace}
    \langle \tr_n(V, \partial \Omega), \varphi \rangle := \int_{\Omega} \nabla \varphi \cdot V \, d  x + \int_{\Omega} \varphi \, d  (\div V) \qquad \forall\varphi \in C^\infty_c(\R^d).
\end{equation}
Assume now that $\Omega$ has a Lipschitz boundary oriented by the outer unit normal, denoted by $n_{\partial \Omega}$. If $V \in C^1(\overline{ \Omega};\R^d)$, then $\tr_n(V, \partial \Omega)$ is induced by the integration in $\partial \Omega$ of $V \cdot n_{\partial \Omega}$. If $V \in \mathcal{MD}^\infty(\Omega)$, then $\tr_n(V, \partial \Omega)$ is represented by a function in $L^\infty(\partial \Omega)$, still denoted by $\tr_n(V, \partial \Omega)$ (see e.g. \cite{ACM05}*{Proposition 3.2}). Moreover, this notion of trace is local in the sense that for any Borel set $\Sigma \subset \partial \Omega_1\cap \partial \Omega_2$ such that $n_{\partial \Omega_1}(x) = n_{\partial \Omega_2}(x)$ for $\mathcal{H}^{d-1}$-a.e. $x \in \Sigma$, where $\Omega_1, \Omega_2$ are Lipschitz open sets contained in $\Omega$, we find that the two traces coincide, i.e. $\tr_n(V, \partial \Omega_1) = \tr_n(V, \partial \Omega_2)$ for $\mathcal{H}^{d-1}$-a.e. $x \in \Sigma$ (see \cite{ACM05}*{Proposition 3.2}). This property allows for the definition of the distributional normal trace on an oriented Lipschitz hypersurface $\Sigma \subset \Omega$. Choose open sets $\Omega_1, \Omega_2$ such that $\Sigma \subset \partial \Omega_1 \cap \partial \Omega_2$ and $n_\Sigma(x) = n_{\partial \Omega_1}(x) = -n_{\partial \Omega_2}(x)$ for $\mathcal{H}^{d-1}$-a.e. $x \in \Sigma$. Then, we define 
$$\tr_n(V, \Sigma_-) : = \tr_n(V, \partial \Omega_1), \qquad \tr_n(V, \Sigma_+) := -\tr_n(V, \partial \Omega_2) \qquad \text{ on } \Sigma.$$

\begin{proposition}[\cite{ACM05}*{Proposition 3.4}]\label{P:div and normal trace} 
Given $V\in \mathcal{MD}^\infty(\Omega)$, the following holds: 
\begin{enumerate}
    \item $\div V \ll \mathcal{H}^{d-1}$; 
    \item for any Lipschitz hypersurface $\Sigma$ oriented by $n$, it holds 
    \begin{equation} \label{eq: divergence of hypersurface}
        \div V \llcorner \Sigma = [\tr_n (V,\Sigma_+)- \tr_n (V,\Sigma_-)]  \, \mathcal H^{d-1} \llcorner \Sigma. 
    \end{equation}
\end{enumerate}
\end{proposition}

The notion of distributional normal trace is usually too weak to deal with non-linear problems, since it is not stable under composition. To overcome this problem, we recall from \cite{CDIN24}*{Section 2.2} a stronger notion of normal trace, which is inspired by the theory of $BV$ functions. See also  \cites{DRINV23,CDIN24,AFP00}. 

\begin{definition}
Let $\Omega \subset \R^d$ be an open set,  let $v \in L^\infty(\Omega ;\R^m)$ and let $\Sigma \subset \Omega$ be a Lipschitz hypersurface oriented by  $n_{\Sigma}$. We say that $v$ has bilateral traces at $\Sigma$ if $\mathcal{H}^{d-1}(\Sigma \cap \mathcal{B}_v^c) = 0 $, that is the triplet $(v^+, v^-, n)$ is well-defined $\mathcal{H}^{d-1}$-a.e. at $\Sigma$. Since $\J_v$ is $(d-1)$-rectifiable, we have $n = n_\Sigma$ at $\mathcal{H}^{d-1}$-a.e. point in $\Sigma \cap \J_v$ (up to changing $n$ with $-n$ and $v^+$ with $v^-$). We say that $v^+, v^-$ are the one-sided traces of $v$ in $\Sigma$. 
\end{definition}

By \cite{CDIN24}*{Theorem 1.4}, the normal component of the one-sided Lebesgue trace agrees with the distributional normal one, whenever both exist\footnote{In \cite{CDIN24}, \cref{t: distributional trace vs lebesgue trace} is shown for vector fields possessing a strong \emph{normal} trace at $\Sigma$, rather than the full ones.}.  

\begin{proposition} \label{t: distributional trace vs lebesgue trace}
Given an open set $\Omega \subset \R^d$, let $V \in \mathcal{MD}^\infty(\Omega)$ and let $\Sigma \subset \Omega$ be a Lipschitz hypersurface oriented by $n_{\Sigma}$. Assume that $V$ has one-sided traces at $\Sigma$. Then, it holds 
$$ \tr_{n}(V, \Sigma_{\pm}) = V^{\pm}\cdot n_\Sigma.$$
\end{proposition}

\subsection{On the convergence of convolutions}

We work with radial convolution kernels, namely
\begin{equation} \label{eq: kernels}
\mathcal{K}_{\rm rad} := \left\{ \eta \in C^\infty_c(B_1; [0,1]) \, : \, \int_{B_1} \eta(z) \, dz =1,  \eta \text{ radial} \right\}. 
\end{equation}
We study the pointwise convergence of the convolutions at continuity and jump points. We recall the following lemma. To give a self contained presentation, we include the short proof.  

\begin{lemma}  [\cite{Inv25}*{Lemma 4.3}] \label{l: averaged convergence}
Let $g\in L^\infty(\R^d; \R^N)$ and fix a kernel $\eta \in \mathcal{K}_{\rm rad}$. Set $\R^d_{\pm} : = \{ \pm z \cdot e_d \geq 0 \}$. Fix $x_0 \in \mathcal{B}_g$ such that \eqref{eq: blow up vector} holds. Then we have 
\begin{equation} \label{eq: shifted blow up}
    \lim_{\ell \to 0} \int_{B_1} g( x_0 + \ell (y-z) ) \, \eta(z)\, dz = \left( \int_{B_1(y) \cap \R^d_+ } \eta(y-z) \, dz \right) g^+ + \left( \int_{B_1(y) \cap \R^d_- } \eta(y - z) \, dz \right) g^- \qquad \forall y \in B_1.
\end{equation}
\end{lemma}

\begin{proof}
Without loss of generality, we assume $x_0 =0$ and $n = e_d$\footnote{Since $\eta$ is radial, the right hand side in \eqref{eq: shifted blow up} is also invariant under rotation.}. Then we have 
\begin{equation}
    \int_{B_1} g(\ell (y-z)) \, \eta(z) \, dz = \left( \int_{B_1(y) \cap \R^d_+} + \int_{B_1(y) \cap \R^d_-}\right) g(\ell z) \, \eta(y-z)\, dz.  
\end{equation}
We compute both terms separately. We have 
\begin{align}
    \abs{ \int_{B_1(y) \cap \R^d_+} g(\ell z) \, \eta(y-z) \, dz - \left( \int_{B_1(y) \cap \R^d_+ } \eta(y-z) \, dz \right) g^+  } & \leq \int_{B_1(y)\cap \R^d_+} \abs{g(\ell z) - g^+} \abs{\eta(y-z)} \, dz . 
\end{align}
By \eqref{eq: blow up of composition}, the latter vanishes in the limit $\ell \to 0$. Similarly, we show that 
$$\lim_{\ell \to 0} \int_{B_1(y) \cap \R^d_-} g(\ell z) \, \eta(y-z) \, dz = \left( \int_{B_1(y) \cap \R^d_- } \eta(y-z) \, dz \right) g^-. $$
\end{proof}

A particular case of the following proposition has been considered in \cite{Inv25}*{Proposition 4.4} with $F: \R^d \to \R^{d\times d}$ defined by $F(b) = b\otimes b$, with $b \in \R^d$. Here, we consider an arbitrary continuous function $F$. 

\begin{proposition} [Double mollification] \label{P: double mollification}
Let $F: \R^N \to \R^M$ be a continuous function, let $g: \R^d \to \R^N$ be a bounded Borel function, and fix a kernel $\eta \in \mathcal{K}_{\rm rad}$. Fix $x_0 \in \mathcal{B}_g$ such that \eqref{eq: blow up vector} is satisfied. Then we have 
\begin{equation} \label{eq: double mollification formula}
    \lim_{\ell \to 0} F(g*\eta_\ell) * \eta_{\ell} (x_0) = \int_0^1 F(s g^+(x_0) + (1-s) g^-(x_0)) \, ds.   
\end{equation}
\end{proposition}

\begin{proof}
Without loss of generality, we assume $x_0 = 0$ and denote by $g^{\pm} := g^{\pm}(0)$. Then we have 
\begin{align}
    F(g* \eta_\ell) * \eta_\ell (0) & = \int_{B_1} F \left( \int_{B_1} g(\ell (y-z) ) \, \eta(z) \, dz \right) \eta(y) \, dy 
\end{align}
By \cref{l: averaged convergence}, we have 
\begin{equation}
    \lim_{\ell \to 0} \int_{B_1} g(\ell (y-z) ) \, \eta(z) \, dz = \alpha(y) g^+ + (1-\alpha(y)) g^-, 
\end{equation}
where we set 
\begin{equation}
    \alpha(y) := \int_{B_1 \cap \{ z_d\leq y_d\}} \eta(z) \, dz = \int_{-1}^{y_d} \left( \int_{B_1 \cap \{ z_d = s \} } \eta(z) \, d \mathcal{H}^{d-1}(z) \right) \, ds.  
\end{equation}
The last identity follows by Fubini's theorem. The function $y \mapsto \alpha (y)$ depends only on the last coordinate $y_d$. Moreover, it is strictly increasing with respect to $y_d$ and satisfies 
\begin{equation}
    \alpha(-1) = 0, \qquad \alpha(1) =1, \qquad \alpha'(y_d) = \int_{B_1\cap \{ z_d= y_d \} } \eta(z) \, dz.  
\end{equation}
Therefore, by the dominated convergence theorem, we write 
\begin{align}
    \lim_{\ell \to 0} F(g* \eta_\ell) * \eta_\ell (0) & = \int_{B_1} F \left(  \alpha(y_d) g^+ +  (1-\alpha(y_d)) g^-  \right) \eta(y) \, dy 
    \\ & = \int_{-1}^1 F \left(  \alpha(y_d) g^+ +  (1-\alpha(y_d)) g^-  \right) \left( \int_{ B_1 \cap  \{z_d =y_d \} } \eta(z) \, d \mathcal{H}^{d-1} (z) \right) \, d y_d 
    \\ & = \int_{0}^1 F( s g^+ + (1-s) g^- )\, dt, 
\end{align}
after the change of variable $s = \alpha(y_d)$. 
\end{proof}

\begin{proposition} \label{P: mollification of function and measure}
Let $F: \R^N \to \R^M$ be a continuous function and let $g: \R^d \to \R^N$ be a bounded Borel function. Let $\lambda$ be a Radon measure on $\R^d$ and assume that the  complement of the set of points where \eqref{eq: blow up vector} holds is $\abs{\lambda}$-negligible. Fix a kernel $\eta \in \mathcal{K}_{\rm rad}$ and set $g_\ell = g* \eta_\ell, \lambda_\ell = \lambda * \eta_\ell$. Then we have 
\begin{equation}
    \lim_{\ell \to 0} F( g_\ell ) \, \lambda_\ell = \left( \int_0^1 F (t g^+ + (1-t) g^-) \, dt \right) \lambda \qquad \text{ weakly in the sense of measures.}
\end{equation}
\end{proposition}

\begin{proof}
Let $\phi \in C_c(\R^d)$ be a test function. By the properties of convolution, we have 
\begin{align}
    \langle\phi, F (g_\ell) \, \lambda_\ell \rangle & = \langle ( F(g_\ell ) \, \phi) * \eta_\ell , \lambda \rangle 
    \\ & = \langle \phi \, (F(g_\ell) * \eta_\ell) , \lambda \rangle + \langle ( \phi \, F(g_\ell) )* \eta_\ell - \phi \, ( F(g_\ell)*\eta_\ell) , \lambda  \rangle = I_\ell+II_\ell.    
\end{align}
Since \eqref{eq: blow up vector} holds $\abs{\lambda}$-a.e., by \cref{P: double mollification} and the dominated convergence theorem, we have 
\begin{equation}
    \lim_{\ell \to 0} I_\ell = \int_{\R^d} \left( \int_0^1 F(s g^+(x) + (1-s)g^-(x)) \, ds \right) \phi(x) \, d\lambda. 
\end{equation}
Regarding the second term, we write explicitly the convolution 
\begin{align}
    \sup_{x \in \R^d} \abs{ ( \phi \, F(g_\ell)  )* \eta_\ell - \phi \, (F(g_\ell)*\eta_\ell)  } & \leq \sup_{x\in \R^d} \int_{B_1} \abs{ \phi( x-\ell y ) - \phi(x) } \abs{ F (g_\ell (x-\ell y) ) \eta(y) } \, dy \lesssim \omega_\phi(\ell), 
\end{align}
where the implicit constant depends only on the supremum norm of $F$ in the range of $g$ and $\omega_\phi$ is the modulus of continuity of $\phi$. Taking $R>1$ such that $\phi$ has compact support in $B_R$, we have 
\begin{equation}
    \lim_{\ell \to 0} II_\ell \lesssim \abs{\lambda}(B_R) \lim_{\ell \to 0} \omega_\phi (\ell)  =0.   
\end{equation}
\end{proof}

\subsection{Functions of bounded variation} \label{ss: BV} 
We list some basic properties of $BV$ functions that will be needed throughout the manuscript. See the monograph \cite{AFP00}. Given an open set $\Omega \subset \R^d$, we say that $g$ has \emph{bounded variation in $\Omega$} ($g \in BV(\Omega)$) if $g \in L^1(\Omega)$ and the distributional gradient $\nabla g$ is represented by a Radon measure on $\Omega$. In the time dependent case, we say that $g \in L^1_t BV_x$ if $g \in L^1_{t,x}$, $g_t \in BV_x$ and $\norm{\nabla g_t}_{\mathcal{M}_x} \in L^1_t$. Functions with bounded variation are characterised by the fact that the difference quotients are uniformly bounded in $L^1$.  More precisely, given $g \in L^1(\R^d)$ we have\footnote{A similar property holds on general open sets.  \label{fn: local characterization}} 
\begin{equation} \label{eq: char of BV}
g \in BV(\R^d) \qquad \Longleftrightarrow \qquad \sup_{h \in \R^d \setminus \{0\} } \frac{\norm{g(\cdot + h) - g(\cdot)}_{L^1(\R^d)}}{\abs{h}} < \infty. 
\end{equation}
Unlike the $W^{1,1}$ case, the jump set of a $BV$ function may have a positive $\mathcal{H}^{d-1}$ measure. See also \cite{DRINV23} for further details. The following is part of the structure theorem for the gradient of $BV$ functions. 

\begin{proposition} [\cite{AFP00}*{Lemma 3.76, Theorem 3.78}] \label{P: prop BV}
Let $g \in BV(\R^d; \R^N)$. The following facts hold: 
\begin{enumerate}
    \item $\abs{\nabla g} \ll \mathcal{H}^{d-1}$; 
    \item the set $\mathcal{B}_g^c$ is $\mathcal{H}^{d-1}$-negligible; 
    \item for any $\mathcal{H}^{d-1}$-countably rectifiable set $\Sigma$, the following holds at $\mathcal{H}^{d-1}$-a.e. point in $\Sigma$: the triplet $(g^+, g^-, n)$ is well defined, \eqref{eq: blow up vector} holds and $n$ is a unit normal vector to $\Sigma$. Moreover, the restriction of the gradient to $\Sigma$ satisfies 
    \begin{equation}
        (\nabla g) \llcorner \Sigma =  (g^+ - g^-) \otimes n  \,    \mathcal{H}^{d-1} \llcorner \Sigma.  
    \end{equation}   
\end{enumerate} 
\end{proposition}

\subsection{Vector fields with bounded deformation} \label{ss: BD} 

We recall some fundamental properties of $BD$ vector fields. We refer to \cites{ACDalM97, TS80} for a detailed discussion. The basic theory is rather similar to that of $BV$ functions, with few exceptions. Given an open set $\Omega \subset \R^d$, we say that $g$ has \emph{bounded deformation in $\Omega$} if $g \in L^1(\Omega)$ and the symmetric distributional gradient $Eg$ is a Radon measure in $\Omega$. In the time dependent case, we say that $g \in L^1_t BD_x$ if $g \in L^1_{t,x}$, $g_t \in BD_x$ and $\norm{E g_t}_{\mathcal{M}_x} \in L^1_t$. Bounded deformation fields enjoy the following characterisation in terms of \emph{longitudinal} difference quotients. Given $g \in L^1(\R^d)$, we have\footref{fn: local characterization}  
\begin{equation} \label{eq: char of BD}
    g \in BD(\R^d) \qquad \Longleftrightarrow \qquad \sup_{h \in \R^d \setminus \{0\}} \frac{\norm{\langle h , g(\cdot + h) - g(\cdot)\rangle}_{L^1(\R^d)}}{\abs{h}^2} < \infty.  
\end{equation}
See e.g. \cite{DRInvN25}*{Lemma 2.11} for further details. The following is part of the structure theorem for the symmetric gradient of $BD$ fields and is analogous to \cref{P: prop BV} for $BV$ functions.  

\begin{proposition} \cite{ACDalM97}*{Theorem 6.1, Equation (4.2)} \label{P: prop BD}
Let $g \in BD(\Omega; \R^d)$. The following facts hold: 
\begin{enumerate}
    \item $\abs{Eg} \ll \mathcal{H}^{d-1}$; 
    \item for any vector field $f \in BD(\Omega; \R^d)$ the set $\mathcal{B}_g^c$ is $\abs{E f}$-negligible\footnote{To the best of the author's knowledge, it is currently not known whether $\mathcal{H}^{d-1}(\S_g \setminus \J_g) =0$, as in the $BV$ setting.}; 
    \item for any $\mathcal{H}^{d-1}$-countably rectifiable set $\Sigma$, the following holds at $\mathcal{H}^{d-1}$-a.e. point in $\Sigma$: the triplet $(g^+, g^-, n)$ is well defined, \eqref{eq: blow up vector} holds, and $n$ is a unit normal vector to $\Sigma$. Moreover, the restriction of the symmetric gradient to $\Sigma$ satisfies 
    \begin{equation} \label{eq: symmetric gradient jump} 
        (E g) \llcorner \Sigma = \frac{(g^+ - g^-) \otimes n + n \otimes (g^+ - g^-)}{2} \, \mathcal{H}^{d-1} \llcorner \Sigma.  
    \end{equation}
\end{enumerate} 
\end{proposition}

\begin{remark} \label{R: negligible set of BD BV} 
As byproducts of \cref{P: prop BV}, \cref{P: prop BD} and \cref{L: product rule BV BD}, we obtain the following fact, which will be used repeatedly throughout the paper. Fix a vector field $g \in L^\infty\cap BD$ and a scalar function $f \in L^\infty \cap BV$. The set $\mathcal{B}_g$ is $\abs{\nabla f}$-negligible. Indeed, letting $h = (f, 0, \dots, 0) \in \R^d$, it is easily checked that $\abs{\nabla \rho} \ll \abs{E h}$ and, by \cref{P: prop BD}, we have $\abs{E h} (\mathcal{B}_g^c) = 0$. 
\end{remark}

We will need the following standard  property on the product of $BV$ functions with $BD$ fields. For the reader's convenience, we include a proof. 

\begin{lemma} \label{L: product rule BV BD}
Let $\rho \in L^\infty \cap BV (\R^d; \R)$ and $u \in L^\infty\cap BD(\R^d; \R^d)$. Then, $\rho u \in L^\infty \cap BD (\R^d; \R^d)$ and the product rule holds 
\begin{equation} \label{eq: chain rule BV BD}
    E(\rho u) = \frac{\rho^+ + \rho^-}{2} \, E u + \frac{(u^+ + u^-) \otimes \nabla \rho + \nabla \rho \otimes (u^+ + u^-) }{4}. 
\end{equation}
In particular, if $\div(u) =0$, we have
\begin{equation} \label{eq: chain rule div BV BD}
    \div(\rho u) = \frac{u^+ + u^-}{2} \cdot \nabla \rho
\end{equation}
\end{lemma}

We point out that the right hand side in \eqref{eq: chain rule BV BD} is a well defined Radon measure. Indeed, by \cref{P: prop BV} and \cref{P: prop BD}, $(\rho^+, \rho^-)$ are defined $\abs{Eu}$-almost everywhere and, also using \cref{R: negligible set of BD BV}, we have that $(u^+, u^-)$ are defined $\abs{\nabla \rho}$-almost everywhere. 

\begin{proof} 
Fix a kernel $\eta \in \mathcal{K}_{\rm rad}$ and let $\rho_\ell = \rho * \eta_\ell, u_\ell = u * \eta_\ell$. Then $E(\rho u)$ is the distributional limit of $E(\rho_\ell u_\ell)$ and we have 
\begin{equation}
    E( \rho_\ell u_\ell) = \rho_\ell \, ( E u_\ell ) + \frac{ u_\ell \otimes \nabla \rho_\ell + \nabla \rho_\ell \otimes u_\ell}{2}. 
\end{equation}
Since $\mathcal{H}^{d-1}(\mathcal{B}_\rho^c) =0$ and $\abs{E u} \ll \mathcal{H}^{d-1}$, we apply \cref{P: mollification of function and measure} with $F(c) = c, c \in \R$ and find
\begin{equation}
    \lim_{\ell \to 0} \rho_\ell  \, (E u_\ell) = \left(\int_0^1 F(s \rho^+ + (1-s) \rho^-) \, ds \right) Eu = \frac{\rho^+ + \rho^-}{2} \, Eu, 
\end{equation}
where the limit is taken weakly in the sense of measures. Similarly, by \cref{R: negligible set of BD BV} we have $ \abs{\nabla \rho} (\mathcal{B}_u^c ) = 0$. Hence, using again \cref{P: mollification of function and measure} and recalling that \eqref{eq: blow up vector} holds $\mathcal{H}^{d-1}$-a.e., we infer that 
\begin{equation}
    \lim_{\ell \to 0} \frac{ u_\ell  \otimes \nabla \rho_\ell + \nabla \rho_\ell \otimes u_\ell}{2} = \frac{(u^+ + u^-) \otimes \nabla \rho + \nabla \rho \otimes (u^+ + u^-) }{4}
\end{equation}
weakly in the sense of measures. Thus, $E(\rho u)$ is a Radon measure and the representation formula \eqref{eq: chain rule BV BD} is valid. To conclude, \eqref{eq: chain rule div BV BD} follows by taking the trace in \eqref{eq: chain rule BV BD}. 
\end{proof}

\section{Proofs} \label{s: proofs}

In this section, we discuss the proofs of our main results. 

\subsection{Dissipation on codimension one hypersurfaces} 

\begin{proof} [Proof of \cref{T: dissipation on hypersurfaces}]
We write $\Sigma = N \cup \bigcup_{i\in \N} \Sigma_i$, where $\mathcal{H}^d(N) = 0$ and $\Sigma_i$ are Lipschitz hypersurfaces in $\Omega \times (0,T)$. By \eqref{eq: defect distribution}, $\D$ is a space-time divergence, namely
$$ - \D = \div_{t,x} \mathcal{V}, \qquad \mathcal{V} := \left( \frac{\rho \abs{u}^2}{2} + \mathcal{P}(\rho) , u \left( \frac{\rho\abs{u}^2}{2} + p(\rho) + \mathcal{P}(\rho) \right) \right) \in L^\infty_{t,x}. $$
By \cref{P:div and normal trace}, we can neglect any set which is $\mathcal{H}^d$-negligible. In particular, we have $\D (N) =0$. Then, fix an index $i \in \N$ and let $\rho^\pm, u^\pm$ be the traces of $\rho,u$ at $\Sigma_i$. Using \cref{t: distributional trace vs lebesgue trace} and \cref{R: blow up of composition} to compute the distributional normal traces of $\mathcal{V}$ in terms of $\rho^\pm, u^\pm$, by \cref{P:div and normal trace} we have  
\begin{align}
    - \D \llcorner\Sigma_i & = (\mathcal{V}^+ - \mathcal{V}^-)\cdot n \, \mathcal{H}^d \llcorner \Sigma_i 
    \\ & = \left[ \frac{\rho^+ \abs{u^+}^2}{2} + \P(\rho^+) -  \frac{\rho^- \abs{u^-}^2}{2} - \P(\rho^-) \right] n_t \, \mathcal{H}^d \llcorner \Sigma_i
    \\ & \qquad + \left[ u^+ \left( \frac{\rho^+ \abs{u^+}^2}{2} + p(\rho^+) + \P(\rho^+) \right) - u^- \left( \frac{\rho^- \abs{u^-}^2}{2} + p(\rho^-) + \P(\rho^-) \right) \right]\cdot n_x \, \mathcal{H}^d \llcorner \Sigma_i, \label{eq: formula V}
\end{align}
where we set $n= (n_t, n_x) \in \R \times \R^d$. In the following, we denote by $u_n^\pm = u^\pm \cdot n_x $. Similarly, we interpret \eqref{compr Euler} as the vanishing condition on the space-time divergence of certain bounded vector fields and, by the same reasoning, for $\mathcal{H}^d$-a.e. point at $\Sigma_i$ we have
\begin{align}
(\rho^+  n_t + \rho^+ u^+_n) u^+ + p(\rho^+) n_x & = (\rho^- n_t + \rho^-  u^-_n ) u^- + p(\rho^-) n_x, \label{eq: momentum traces}
\\ \rho^+ n_t + \rho^+ u^+_n & = \rho^- n_t + \rho^- u^-_n \label{eq: continuity traces}. 
\end{align} 
Taking the scalar product of \eqref{eq: momentum traces} by $u^+ + u^-$ and using \eqref{eq: continuity traces}, we find 
\begin{align}
    0 & = (\rho^+ n_t + \rho^+ u^+_n) \frac{\abs{u^+}^2 }{2} - (\rho^- n_t + \rho^- u^-_n) \frac{\abs{u^-}^2}{2}  + (p(\rho^+) - p(\rho^-) ) \frac{u^+_n + u^-_n}{2} \label{eq: manipulation 1}
\end{align}
$\mathcal{H}^d$-a.e. at $\Sigma_i$. Plugging \eqref{eq: manipulation 1} into \eqref{eq: formula V}, we find 
\begin{align}
    (\mathcal{V}^+ - \mathcal{V}^- )\cdot n & = - \frac{p(\rho^+) - p(\rho^-)}{2} (u^+_n + u^-_n) + p(\rho^+) u^+_n - p(\rho^-) u^-_n + \P( \rho^+) (n_t + u^+_n) - \P(\rho^-) (n_t + u^-_n)
    \\ & = \frac{p(\rho^+) + p(\rho^-)}{2} (u^+_n - u^-_n) + ( n_t +  u^+_n) \P(\rho^+) - ( n_t +  u^-_n) \P(\rho^-)
\end{align}
at $\mathcal{H}^d$-a.e. point at $\Sigma_i$. Using \eqref{internal energy} and \eqref{eq: continuity traces}, for $\rho^+ \neq \rho^-, \rho^+ > 0$ and $\rho^->0$, we write 
\begin{align}
    ( n_t +  u^+_n) \P(\rho^+) - ( n_t +  u^-_n) \P(\rho^-) & = (\rho^+ (n_t + u^+_n)) \int_{\rho^-}^{\rho^+} \frac{p(r)}{r^2} \, dr
    \\ & = (\rho^+ - \rho^-) (\rho^+ (n_t + u^+_n)) \fint_{\rho^-}^{\rho^+} \frac{p(r)}{r^2}\, dr 
    \\ & = \left[ \rho^+ \rho^- (n_t + u^-_n) - \rho^+ \rho^- (n_t + u^+_n) \right] \fint_{\rho^-}^{\rho^+} \frac{p(r)}{r^2}\, dr 
    \\ & = - (u^+_n - u^-_n)\frac{\rho^+ \rho^-}{(\rho^+ - \rho^-)} \int_{\rho^-}^{\rho^+} \frac{p(r)}{r^2} \, dr. 
\end{align}
Now, suppose that $\rho^+ = \rho^- >0$. By \eqref{eq: continuity traces} and \eqref{internal energy} we have   
\begin{align}
    (n_t + u^+_n) \P(\rho^+) = (n_t + u^-_n) \P(\rho^-), \qquad  u^+_n = u^-_n.
\end{align}
Thus we have $ (\mathcal{V}^+ - \mathcal{V}^-)\cdot n = 0$. The case in which $\rho^+ = 0$ or $\rho^- =0$ can be addressed in the same way. 
\end{proof}

\subsection{The local energy balance}

\begin{proof} [Proof of \cref{t: main DR}] 
For the sake of clarity, we break down the proof into steps. We adopt the notation of \eqref{eq: favre mollification}. For simplicity, we set $\overline{p} = \overline{p(\rho)}$. 

\textsc{\underline{Step 1}:} Mollifying the momentum and the mass equations, we obtain 
    \begin{align} 
    \partial_t \overline{\rho u} + \div \left( \overline{\rho u \otimes u} \right) + \nabla \overline{p} & = \overline{\rho f}, \label{eq: en bal 1}
    \\ \partial_t \overline{\rho} + \div(\overline{\rho u}) & =0. \label{eq: en bal 1.5}
\end{align}
By the properties of convolution, $\overline{\rho}, \overline{\rho u}$ are smooth in space and, by \eqref{eq: en bal 1} and \eqref{eq: en bal 1.5}, it turns out that $\partial_t \overline{\rho}, \partial_t \overline{\rho u} \in L^\infty_{t,x}$. Therefore, $\overline{\rho}, \overline{\rho u}$ are Lipschitz continuous in space-time and the following computations are justified. We aim to take the scalar product of \eqref{eq: en bal 1} by $\widetilde{u}$. To this end, by \eqref{eq: en bal 1.5}, we compute 
\begin{align} 
    \widetilde{u} \cdot \partial_t (\overline{\rho u}) & = \widetilde{u} \cdot \partial_t(\overline{\rho}\, \widetilde{u}) = \abs{\widetilde{u}}^2 \partial_t \overline{\rho} + \overline{\rho} \, \widetilde{u} \cdot \partial_t \widetilde{u}
    \\ & = \partial_t \left( \frac{\overline{\rho} \abs{\widetilde{u}}^2}{2}  \right) + \frac{\abs{\widetilde{u}}^2}{2} \partial_t \overline{\rho} = \partial_t \left( \frac{\overline{\rho} \abs{\widetilde{u}}^2}{2}  \right) - \frac{\abs{\widetilde{u}}^2}{2} \div(\overline{\rho u}), \label{eq: en bal 2} 
\end{align}
\begin{align}
    \widetilde{u}\cdot \div (\overline{\rho u \otimes u}) & = \widetilde{u} \cdot \div \left( \overline{\rho} \, (\widetilde{u\otimes u} - \widetilde{u} \otimes \widetilde{u}) \right) + \widetilde{u} \cdot (\overline{\rho} \,\widetilde{u} \cdot \nabla) \, \widetilde{u} + \div (\overline{\rho u}) \abs{\widetilde{u}}^2
    \\ & = \widetilde{u} \cdot \div \left( \overline{\rho} \, (\widetilde{u\otimes u} - \widetilde{u} \otimes \widetilde{u}) \right) + \abs{\widetilde{u}}^2 \overline{\rho u} \cdot \nabla \frac{\abs{\widetilde{u}}^2}{2} + \div (\overline{\rho u}). \label{eq: en bal 3} 
\end{align}
Therefore, taking the scalar product of \eqref{eq: en bal 1.5} by $\widetilde{u}$ and using \eqref{eq: en bal 2} and \eqref{eq: en bal 3}, we obtain 
\begin{equation}
    \partial_t \left( \frac{\overline{\rho} \abs{\widetilde{u}}^2}{2}\right) + \div \left( \overline{\rho u} \, \frac{\abs{\widetilde{u}}^2}{2} \right) + \div \left( \overline{\rho} \, (\widetilde{u\otimes u} - \widetilde{u} \otimes \widetilde{u}) \right) \cdot \widetilde{u} + \nabla \overline{p} \cdot \widetilde{u} = \overline{\rho f} \cdot \widetilde{u}. 
\end{equation}
Further manipulations yield 
\begin{align}
    \partial_t \left( \frac{\overline{\rho} \abs{\widetilde{u}}^2}{2}\right) + \div \left( \widetilde{u} \left(  \frac{ \overline{\rho} \abs{\widetilde{u}}^2}{2} + \overline{p} \right) + \overline{\rho u} \, (\widetilde{u\otimes u} - \widetilde{u} \otimes \widetilde{u} ) \right) & =  \overline{p} \, \div (\widetilde{u}) + \overline{\rho f} \cdot \widetilde{u} + \overline{\rho} \, (\widetilde{u \otimes u} - \widetilde{u}\otimes \widetilde{u}) : \nabla \widetilde{u}
    \\ & = \overline{p} \, \div (\widetilde{u}) + \overline{\rho f} \cdot \widetilde{u} - J_1^\ell, \label{eq: en bal 5}
\end{align}
since the matrix $\widetilde{u \otimes u} -\widetilde{u}\otimes \widetilde{u}$ is symmetric and  $J_1^\ell$ is defined by \eqref{eq: J_1}. 

\textsc{\underline{Step 2}:} To study the evolution of $\P(\rho)$, we multiply \eqref{eq: en bal 1.5} by $\P'(\overline{\rho})$ and we find  
\begin{align}
     \partial_t \P(\overline{\rho}) + \P'(\overline{\rho}) \div (\overline{\rho} \overline{u}) = 0. \label{eq: en bal 6}
\end{align}
Observe that $\P$ satisfies
\begin{equation} \label{eq: prop of P}
    z \, \P'(z) = \P(z) +p(z) \qquad \forall z >0. 
\end{equation}
Standard manipulations give 
\begin{align}
    \P'(\overline{\rho}) \div(\overline{\rho u}) & = \P'(\overline{\rho}) \, \overline{\rho} \div(\widetilde{u}) + \P'(\overline{\rho}) \, \nabla\overline{\rho} \cdot \widetilde{u} = \div(\P(\overline{\rho}) \, \widetilde{u}) + p(\overline{\rho}) \div(\widetilde{u}). \label{eq: en bal 7}
\end{align}
Summing \eqref{eq: en bal 5} and \eqref{eq: en bal 6} and using \eqref{eq: en bal 7}, we have 
\begin{align}
    \partial_t \left( \frac{\overline{\rho} \abs{\widetilde{u}}^2}{2} + \P(\overline{\rho}) \right) & + \div \left( \widetilde{u} \left( \frac{ \overline{\rho} \abs{\widetilde{u}}^2}{2} + \overline{p(\rho)} + \P(\overline{\rho}) \right) + \overline{\rho u} \, (\widetilde{u\otimes u} - \widetilde{u} \otimes \widetilde{u} ) \right)  = \overline{\rho f} \cdot \widetilde{u} - J_1^\ell - J_2^\ell,   
\end{align}
where $J_2^\ell$ is defined by \eqref{eq: J_2}.

\textsc{\underline{Step 3}:} Letting $\ell \to 0$, by the assumptions on $u,\rho,f$, the left hand side goes to 
$$ \partial_t \left( \frac{\rho \abs{u}^2}{2} + \P(\rho)  \right) + \div \left( u \left( \frac{\rho \abs{u}^2}{2} + p(\rho) + \P(\rho) \right) \right) $$ 
in the sense of distributions and $\overline{\rho f} \cdot \widetilde{u} \to \rho f \cdot u$ in $L^1_{t,x}$. Therefore, we conclude that $J_1^\ell +J_2^\ell \rightharpoonup \D$ in the sense of distributions. 
\end{proof}

\subsection{The bounded variation case}

To achieve the proof of \cref{T: main bounded var} we explicitly compute the limit of the error terms given by \cref{t: main DR}. We follow the blow up analysis of \cite{Inv25}. We treat separately the terms $J_1^\ell, J_2^\ell$ and we begin with the stationary case, the time dependent one being a minor modification on the following argument. 

\begin{proposition} \label{P: J^1}
Let $u \in L^\infty \cap BD$ and let $\rho \in L^\infty \cap BV$ satisfy \eqref{no vacuum}. Fix $\eta \in \mathcal{K}_{\text{rad}}$ and for $\ell>0$ define $J_1^\ell$ by \eqref{eq: J_1}, that is
\begin{equation}
    J_1^\ell = \overline{\rho} \left( \widetilde{u} \otimes \widetilde{u} - \widetilde{u \otimes u} \right) : E \widetilde{u}. 
\end{equation}  
Let $\rho^\pm, u^\pm$ be the traces of $\rho,u$ in $\J_u$, which is oriented by a unit normal vector $n$. Then we have
\begin{equation} \label{eq: formula J^1 explicit}
    \lim_{\ell \to 0} J_1^\ell = - \left[ \abs{\rho^+ \rho^-}^2 (u^+_n - u^-_n) \abs{u^+-u^-}^2 \int_0^1 \frac{s(1-s)}{(s \rho^+ + (1-s) \rho^-)^3} \, ds \right] \mathcal{H}^{d-1} \llcorner \J_u
\end{equation} 
weakly in the sense of measures, where we set $u^\pm_n = u^\pm \cdot n$. 
\end{proposition}

\begin{proof}
By \eqref{eq: favre mollification} we have 
\begin{equation}
    J_1^\ell =  \left[ \frac{\overline{\rho u} \otimes \overline{\rho u} }{(\overline{\rho})^2} - \frac{\overline{\rho u \otimes u }}{\overline{\rho}} \right] : E ( \overline{\rho u}) - \left[ \frac{(\overline{\rho u}\otimes \overline{\rho u} ) \, \overline{\rho u}}{(\overline{\rho})^3} - \frac{(\overline{\rho u \otimes u}) \, \overline{\rho u}}{(\overline{\rho})^2}  \right] \cdot \nabla \overline{\rho} =: J_{1,1}^\ell - J_{1,2}^\ell. 
\end{equation}
We treat both terms separately. 

\textsc{\underline{The term $J_{1,1}^\ell$}: } We describe the procedure in detail, since the same argument will be repeated multiple times to analyse all the other terms. We denote by 
\begin{equation}
    F_{1,1}(a,b,c) := \frac{b\otimes b}{c^2} - \frac{a}{c}   \qquad (a,b,c)\in \R^{d\times d} \times \R^d \times \R. 
\end{equation} 
By \cref{L: product rule BV BD}, we have $\rho u \in BD$. Hence, by \cref{P: prop BV} and \cref{P: prop BD}, the set $(\mathcal{B}_u \cap \mathcal{B}_\rho)^c$ is $\abs{E (\rho u)}$-negligible. Setting 
$$ V:= (\rho u \otimes u, \, \rho u, \, \rho ), $$
by \cref{R: blow up of composition}, we infer that  $\mathcal{B}_V^c$ is $\abs{E(\rho u)}$-negligible and 
\begin{equation}
    (V^+, V^-, n_V) = (( \rho^+ u^+ \otimes u^+ , \rho ^+ u^+ , \rho^+), ( \rho^- u^-\otimes u^-, \rho^- u^-, \rho^-), n ) \qquad \abs{E(\rho u)}-\text{a.e.}.  
\end{equation}
Thus, by \cref{P: mollification of function and measure} we have
\begin{equation}
    \lim_{\ell \to 0 } J_{1,1}^\ell = \left( \int_0^1 F_{1,1} (s V^+ + (1-s) V^- ) \, ds \right) : E(\rho u), 
\end{equation}
weakly in the sense of measures. Then, for any $s\in (0,1)$ we compute
\begin{align}
     & F_{1,1}(s V^+ +(1-s) V^-) 
     \\ & \qquad =  \frac{(s\rho^+ u^+ +(1-s) \rho^- u^- ) \otimes (s \rho^+ u^+ + (1-s) \rho^- u^-)}{ (s \rho^+ + (1-s) \rho^-)^2 } -  \frac{ s \rho^+ u^+\otimes u^+ + (1-s) \rho^- u^- \otimes u^- }{\rho^+ s + \rho^-(1-s)} 
     \\ & \qquad = - \rho^+ \rho^- (u^+-u^-)\otimes (u^+-u^-)   \frac{ s(1-s) }{(s \rho^+ + (1-s) \rho^-)^2} . \label{eq: F^1,1 explicit}
\end{align}
The latter vanishes for $x \in \mathcal{S}_u$. Thus, recalling that $\J_u$ is $(d-1)$-countably rectifiable and using \cref{P: prop BD}, we have
\begin{align}
    \lim_{\ell \to 0} J_{1,1}^\ell & = \left( \int_0^1 F_{1,1} (s V^+ + (1-s) V^- ) \, ds \right) : E(\rho u) \llcorner \J_u
    \\ & = \left( \int_0^1 F_{1,1} (s V^+ + (1-s) V^- ) \, ds \right) : \left[ \frac{(\rho^+ u^+ - \rho^- u^- ) \otimes n + n \otimes (\rho^+ u^+ -\rho^- u^-)}{2} \right] \mathcal{H}^{d-1} \llcorner \J_u.  
\end{align}

\textsc{\underline{The term $J_{1,2}^\ell$}: } We denote by 
$$F_{1,2} (a,b,c) := \frac{(b\otimes b) b }{c^3} - \frac{a b}{c^2} = F_{1,1} (a,b,c) \frac{b}{c} \qquad (a,b,c) \in \R^{d\times d} \times \R^d \times \R. $$
Repeating the argument of the previous step with \cref{R: negligible set of BD BV}, we have 
\begin{align}
    \lim_{\ell \to 0 } J_{1,2}^\ell & = \left( \int_0^1 F_{1,2} (s V^+ + (1-s) V^- ) \, ds \right) \cdot  \nabla \rho 
    \\ & =  \left( \int_0^1 F_{1,1}(sV^+ +(1-s) V^- ) : \frac{(s \rho^+ u^+ + (1-s) \rho^- u^-) }{s \rho^+ + (1-s) \rho^-} \otimes n \, ds \right) (\rho^+ -\rho^-) \, \mathcal{H}^{d-1} \llcorner \J_u 
\end{align} 
weakly in the sense of measures, where we recall that $V=(\rho u\otimes u, \, \rho u, \, \rho)$. 

\textsc{\underline{Summing all terms}:} In summary, by \eqref{eq: F^1,1 explicit} and since $(u^+-u^-)\otimes (u^+-u^-)$ is symmetric, we have 
\begin{align}
    \lim_{\ell \to 0} & J_{1,1}^\ell - J_{1,2}^\ell  = - \rho^+ \rho^- (u^+ - u^-)\otimes (u^+ -u^-) 
    \\ & : \left( \int_0^1 \frac{s(1-s) }{(s \rho^+ + (1-s) \rho^-)^2} \left( (\rho^+ u^+ -\rho^- u^-) - \frac{ (s \rho^+ u^+ + (1-s) \rho^- u^-) (\rho^+ - \rho^-)}{s\rho^+ + (1-s) \rho^-} \right) \otimes n \, ds  \right)\mathcal{H}^{d-1}\llcorner \J_u, 
\end{align}
weakly as measures. A direct computation shows that 
\begin{equation}
    (\rho^+ u^+ -\rho^- u^-) - \frac{ (s \rho^+ u^+ + (1-s) \rho^- u^-) (\rho^+ - \rho^-)}{s\rho^+ + (1-s) \rho^-} = \frac{\rho^+ \rho^- (u^+ -u^-)}{s \rho^+ + (1-s)\rho^-}. 
\end{equation}
Hence, \eqref{eq: formula J^1 explicit} follows. 
\end{proof}

\begin{proposition} \label{P: J^2}
Fix a continuous function $p: \R \to \R$. Let $u \in L^\infty \cap BD$ and let $\rho \in L^\infty \cap BV$ satisfy \eqref{no vacuum}.  Fix $\eta \in \mathcal{K}_{\text{rad}}$ and for $\ell>0$ define $J_2^\ell$ by \eqref{eq: J_2}, that is 
\begin{equation}
    J_2^\ell =  \div(\widetilde{u}) \left( p(\overline{\rho}) - \overline{p(\rho)} \right). 
\end{equation}
Let $\rho^\pm, u^\pm$ be the traces of $\rho,u$ in $\J_u$, which is oriented by a unit normal vector $n$. Then we have 
\begin{equation} \label{eq: formula J^2 explicit}
    \lim_{\ell \to 0} J_2^\ell = \left[ \rho^+ \rho^- (u^+_n - u^-_n) \int_0^1 \frac{p(s \rho^+ + (1-s) \rho^-) - (s p(\rho^+) + (1-s) p(\rho^-) ) }{(s \rho^+ + (1-s) \rho^-)^2} \, ds \right] \mathcal{H}^{d-1} \llcorner (\J_\rho \cap \J_u)  
\end{equation}
weakly in the sense of measures, where we set $u^\pm_n = u^\pm \cdot n$.  
\end{proposition}

\begin{proof}
By the definition of $\widetilde{u}, \overline{p(\rho)}, p(\overline{\rho})$, we have 
\begin{equation}
    J_2^\ell = \left[ \frac{ p(\overline{\rho}) - \overline{p(\rho)} }{\overline{\rho}} \right] \div(\overline{\rho u}) - \left[ \frac{ p(\overline{\rho}) - \overline{p(\rho)}}{(\overline{\rho})^2} \right] \overline{\rho u}\cdot \nabla \overline{\rho} =: J_{2,1}^\ell -  J_{2,2}^\ell.   
\end{equation}
As in the proof of \cref{P: J^1}, we estimate both terms separately. 

\textsc{\underline{The term $J_{2,1}^\ell$}: } We denote by 
\begin{equation}
    F_{2,1}(c,d) := \frac{p(c) - d }{c} \qquad (c,d) \in \R\times \R.  
\end{equation}
We combine \cref{L: product rule BV BD}, \cref{P: prop BV}, \cref{P: prop BD}, \cref{R: blow up of composition} as in the proof of \cref{P: J^1}. By \cref{P: mollification of function and measure}, we infer 
\begin{equation}
    \lim_{\ell \to 0} J_{2,1}^\ell = \left( \int_0^1 F_{2,1} (s W^+ + (1-s) W^-) \, ds \right) \div(\rho u), \qquad W:= (\rho, p(\rho) )
\end{equation}
weakly as measures. Since $F_{2,1}( W^+ ) = 0 $ whenever $\rho^+ = \rho^-$ (which implies $W^+ = W^-$), by \cref{P: prop BD} we have  
\begin{equation}
    \lim_{\ell \to 0} J_{2,1}^\ell = \left( \int_0^1 F_{2,1} (s W^+ + (1-s) W^-) \, ds \right) (\rho^+ u^+_n - \rho^- u^-_n) \, \mathcal{H}^{d-1} \llcorner \J_{\rho} 
\end{equation}
weakly as measures. 

\textsc{\underline{The term $J_{2,2}^\ell$}: } We denote by 
\begin{equation}
    F_{2,2} (b,c,d) :=  \frac{p(c) - d}{c^2} b = F_{2,1} (c,d) \frac{b}{c} \qquad (b,c,d) \in \R^d \times \R \times \R.  
\end{equation}
We set $Z := (\rho u, \rho, p(\rho)) = (\rho u, W)$. With the same reasoning as in the proof of the second step of \cref{P: J^1}, we obtain 
\begin{align}
    \lim_{\ell \to 0} J_{2,2}^\ell & = \left( \int_0^1 F_{2,2}(s Z^+ + (1-s) Z^-) \, ds \right) \cdot \nabla \rho
    \\ & = \left( \int_0^1 F_{2,1} (s W^+ + (1-s) W^-) \frac{(s \rho^+ u^+_n + (1-s) \rho^- u^-_n)(\rho^+ - \rho^-)}{s \rho^+ + (1-s) \rho^- u^-} \, ds \right) \mathcal{H}^{d-1} \llcorner \J_\rho,   
\end{align}
where we used $F_{2,1}( W^+ ) = 0 $ whenever $\rho^+ = \rho^-$ and \cref{P: prop BV}. 

\textsc{\underline{Summing all terms}: } In summary, we have 
\begin{align}
    \lim_{\ell \to 0}  J_{2,1}^\ell - J_{2,2}^\ell & = \bigg[ \int_0^1 \frac{  p(s\rho^+ + (1-s) \rho^-) - (s p(\rho^+) +(1-s) p(\rho^-) ) }{s \rho^+ +(1-s)\rho^-} 
    \\ & \qquad \qquad \left( (\rho^+ u^+_n - \rho^- u^-_n) - \frac{(s \rho^+ u^+_n +(1-s) \rho^- u^-_n) (\rho^+ -\rho^-)}{s \rho^+ + (1-s)\rho^-} \right) \, ds \bigg] \mathcal{H}^{d-1}\llcorner \J_{\rho}
    \\ & = \left( \rho^+ \rho^- (u^+_n - u^-_n) \int_0^1 \frac{  p(s\rho^+ + (1-s) \rho^-) - (s p(\rho^+) +(1-s) p(\rho^-)) }{(s \rho^+ +(1-s)\rho^-)^2} \, ds \right) \mathcal{H}^{d-1} \llcorner \J_\rho,   
\end{align}
thus proving \eqref{eq: formula J^2 explicit}. 
\end{proof}

Finally, we are in the position to give the proof of \cref{T: main bounded var}. 

\begin{proof} [Proof of \cref{T: main bounded var}]
Fix a test function $\phi \in C_c(\Omega \times (0,T))$. Denote by $\rho_t^\pm, u_t^\pm$ the traces of $\rho_t, u_t$ at $\J_{u_t}$, which is oriented by a normal unit vector $n$, and $(u_t^\pm)_n = u_t^\pm \cdot n$. By \cref{P: J^1} and \cref{P: J^2}, for a.e. $t$ we have 
\begin{equation}
    \lim_{\ell \to 0} \langle J_1^\ell(t) + J_2^\ell (t), \phi_t \rangle = - \int_{\Omega\cap \J_{u_t}}  \phi_t \left[ \rho_t^+ \rho_t^- ((u^+_t)_n - (u^-_t)_n) \, \Phi_t \right] \, d \mathcal{H}^{d-1}(x), 
\end{equation}
where $\Phi_t$ is given by 
\begin{align}
    \Phi_t & := \rho^+_t \rho^-_t \abs{u^+_t - u^-_t}^2 \int_0^1 \frac{s(1-s)}{(s\rho^+_t + (1-s) \rho^-_t )^3} \, ds + \int_0^1 \frac{s p(\rho^+_t) + (1-s) p(\rho^-_t) -  p(s \rho^+_t + (1-s) \rho^-_t)) }{(s \rho^+_t + (1-s)\rho^-_t)^2}\, ds .  \label{eq: formula density}
\end{align}
Since both sequences $t \mapsto \langle J_1^\ell(t), \phi_t \rangle, t \mapsto \langle J_2^\ell(t), \phi_t\rangle $ are bounded in $L^\infty_t$, by dominated convergence with respect to the time variable we conclude 
\begin{align} \label{eq: formula dissipation preliminary}
    \langle \D , \phi\rangle & = \int_0^T \lim_{\ell \to 0} \langle J_1^\ell(t) + J_2^\ell(t) , \phi_t \rangle \, dt = - \int_0^T \int_{\Omega \cap \J_{u_t}} \phi_t(x) \, [\rho^+_t \rho^-_t((u^+_t)_n -(u^-_t)_n)\, \Phi_t](x) \, d \mathcal{H}^{d-1} \, dt 
    \\ & = - \int_0^T \int_{\J_{u_t}} \phi_t(x) \rho^+_t(x) \rho^-_t(x)  \Phi_t(x) \, d \div(u_t)(x) \, dt.   
\end{align} 
To conclude the proof, we compute explicitly the integrals defining the function $\Phi_t$. For the reader's convenience, this is postponed to the following lemma. 
\end{proof}

\begin{lemma} \label{L: equivalent formulas} 
Let $\Phi, \Gamma, \Psi, \Pi$ be given by \eqref{eq: formula density}, \eqref{eq: function Gamma}, \eqref{eq: function Psi}, and \eqref{eq: function Pi}, respectively. For any $a,b \in (0,+\infty), v,w \in \R^d$ we have 
\begin{align} \label{eq: massaging diss formula}
    a b \, \Phi = \Gamma(a, b, v, w) \, \Psi(a,b) + \Pi (a,b). 
\end{align} 
Furthermore, assume that the pressure law is convex and nondecreasing. Then we have 
\begin{equation}
    \Pi(a,b) \geq 0.  
\end{equation}
\end{lemma}

\begin{proof} 
Since $\Phi$ is symmetric with respect to $a,b$, without loss of generality, we may assume $a \geq b >0$. To begin, we address the case $a  = b $. We have 
\begin{equation}
    ab \, \Phi = a  \abs{ v-w}^2   \int_0^1 s(1-s) \, ds = \frac{a}{6} \abs{v-w}^2.  
\end{equation}
From now on, we assume  $a > b >0$. Then, we explicitly compute the integrals in \eqref{eq: formula density}. We have 
\begin{equation}
    \Phi = ab \, \abs{v-w}^2 A + p(a) \, B_1 + p(b) \, B_2 - B_3,
\end{equation}
where we set 
\begin{align}
    A & := \int_0^1 \frac{s(1-s)}{(s a + (1-s)b)^3} \, ds = \frac{a + b}{2 ab \, (a-b)^2} - \frac{\ln(a) - \ln(b)}{(a-b)^3}, 
    \\ B_1 & : = \int_{0}^1 \frac{s}{(s a + (1-s) b)^2} \, ds = \frac{1}{(a-b)^2} \left[  \ln (a) - \ln (b) - \frac{a-b}{a} \right], 
    \\ B_2 & : = \int_{0}^1 \frac{1-s}{(s a + (1-s) b)^2} \, ds = \frac{1}{(a-b)^2} \left[  - (\ln (a) - \ln (b)) + \frac{a-b}{b} \right], 
    \\ B_3 & := \int_0^1 \frac{p(s a + (1-s) b )}{(s a + (1-s) b)^2} \, ds = \frac{1}{a-b} \int_{b}^{a} \frac{p(r)}{r^2} \, dr, 
    \end{align}
after some elementary (but tedious) computations\footnote{Note that $B_1$ and $B_2$ coincide if we permute $a$ and $b$.}. Thus, we write 
\begin{align}
    p(a) \, B_1 + p(b) \, B_2 = \frac{\ln(a) -\ln(b)}{(a-b)^2} (p(a) - p(b)) - \left( \frac{p(a)}{a} - \frac{p(b)}{b} \right) \frac{1}{a-b}.  
\end{align}
Then, we notice that 
\begin{equation}
    \frac{ab}{a-b} \left( \frac{p(a)}{a} - \frac{p(b)}{b} \right) = \frac{a+b}{2(a-b)} (p(a) - p(b)) - \frac{p(a) + p(b)}{2}. 
\end{equation} 
Therefore, we obtain 
\begin{equation}
    \Pi(a,b) = ab \, \left[ p(a) \, B_1 + p(b) \, B_2 - B_3 \right] + \frac{p(a) - p(b)}{a-b} \left[ \frac{a+b}{2} - \, ab \frac{\ln(a) - \ln(b)}{a-b} \right]. 
\end{equation}
Then \eqref{eq: massaging diss formula} follows. Assume now that the pressure law is nondecreasing and convex. Hence, looking at the integral definitions of $B_1, B_2, B_3$, it is clear that 
$$ p(a) \, B_1 + p(b) \, B_2 - B_3 \geq 0, \qquad \frac{p(a) - p(b)}{a-b} \geq 0, $$
by monotonicity and convexity. It remains to check that 
\begin{equation}
    \frac{a+b}{2} - ab \, \frac{\ln(a) - \ln(b)}{a-b} = \frac{b}{2(z-1)} \left[ z^2 - 1 - 2z \ln (z) \right] \geq 0, \qquad z= \frac{a}{b} >1. 
\end{equation} 
To prove that $\ z^2 - 1 - 2z \ln (z) \geq 0$ for $ z >0 $, it is enough to observe that 
$$\frac{d}{dz} (z^2 -1 - 2 z \ln(z)) = 2 (z-1 - \ln(z)) > 0 \qquad \text{ for } z>1$$
by the concavity of the logarithm. 
\end{proof}

\begin{remark} 
The function $\Psi$ defined by \eqref{eq: function Psi} is continuous on $(0, +\infty)$. Changing $a$ and $b$, we may assume $a \geq b$. By the Taylor expansion, we see that 
\begin{equation}
    \Psi(a,b) = \frac{1}{b (z-1)} \left[ \frac{z+1}{2(z-1)} - \frac{z \ln(z)}{(z-1)^2} \right] = \frac{1}{b} \left[ \frac{1}{2} -\frac{z}{3} + O(\abs{z-1}) \right] \qquad \text{ for } z = \frac{a}{b} \geq 1, \, b >0,
\end{equation}
The reminder is independent of $b$. Thus, it is clear that $\Psi$ is continuous on $(0, +\infty)^2$. 
\end{remark}

\subsection{Back to the formula for the dissipation} \label{ss: comparing formulas}

In this section, we compare the two formulas given by \cref{T: dissipation on hypersurfaces} and \cref{T: main bounded var}. It is clear that the two formulations must coincide at least in the stationary case. The main issue is that the equations provide information on the space-time dynamics. Thus, the result of \cref{T: dissipation on hypersurfaces} is natural from this point of view. On the other hand, exploiting the spatial regularity and fine properties of spatial commutators, in \cref{T: main bounded var} we compute the dissipation on space hypersurfaces, which are defined for almost every time, but do not need to be the time slice of a global space-time hypesurface. Moreover, it is not clear how to connect the existence of \emph{space traces for a.e. time} with that of \emph{space-time traces}. From a technical point of view, the notion of time disintegration of a space-time measure gives a possibility to establish a connection between the two scenarios.  

Fix $\mu \in \mathcal{M}((0,T) \times \Omega)$. We say that $\mu$ has a disintegration with respect to the time variable if there exists a family $\{\mu_t\}_t \subset \mathcal{M}(\Omega)$ such that $\mu = \mu_t \otimes \mathcal{L}^1\llcorner (0,T)$, that is 
\begin{equation} \label{eq: disintegration}
    \int_{(0,T) \times \Omega} \phi \, d \mu = \int_0^T \int_{\Omega} \phi_t \, d\mu_t \, dt \qquad \forall \phi \in C_c((0,T) \times \Omega).
\end{equation}
If it exists, the family $\{ \mu_t\}_t$ is uniquely determined for $\mathcal{L}^1$-a.e. $t \in (0,T)$, in the following sense. If $\{T_t\}_t$ is a weakly measurable family in $\mathcal{D}'(\Omega)$ such that $\mu= T_t \otimes dt $, that is
$$ \int_{(0,T) \times \Omega} \phi \, d \mu = \int_0^T \langle T_t, \phi_t\rangle \, dt \qquad \forall t\in (0,T),  $$
then $T_t = \mu_t$ for a.e. $t \in (0,T)$. Denoting by $\pi_t: (0,T) \times \Omega \to (0,T)$ the projection onto the time variable, it is known that the existence of the disintegration for $\mu$ is equivalent to $(\pi_t)_{\#}\mu \ll \mathcal{L}^1\llcorner (0,T)$. See e.g. \cite{AFP00}*{Theorem 2.28}. 

\begin{proposition} \label{P: equivalent formula}
Let $(\rho, u) \in L^\infty_{t,x}$ be a solution to \eqref{compr Euler} with pressure law $p \in C(\R)$ such that $p(0)=0$ and external force $f \in L^1_{t,x}$. Assume that both $ \partial_t \rho, \partial_t(\rho u) \in \mathcal{M}_{t,x} $ have a disintegration of the form
$$\partial_t \rho = \mu_t \otimes dt, \qquad \partial_t (\rho u) = \nu_t \otimes dt, $$ 
with $\{\mu_t\}_t,  \{\nu_t \}_t \subset \mathcal{M}_x$ satisfying \eqref{eq: disintegration}, and that for $\mathcal{L}^1$-a.e. $t \in (0,T)$ we have
\begin{equation} \label{eq: disintegrating momentum and continuity}
    \mu_t(\Sigma) = 0, \quad \nu_t(\Sigma) =0 \qquad \text{for any Lipschitz hypersurface } \Sigma \subset \Omega.
\end{equation}
Then, the following holds for a.e. $t$. Let $\Sigma \subset \Omega$ be a Lipschitz hypersurface oriented by a unit normal vector $n : \Sigma \to \mathbb{S}^{d-1}$ and assume that both $\rho_t, u_t$ have bilateral traces at $\Sigma$. Set $(u_t^\pm)_n = u_t^\pm \cdot n$. For $\mathcal{H}^{d-1}$-a.e. point at $\Sigma$, we have 
\begin{equation}
    ((u_t^+)_n - (u_t^-)_n) \left[ \rho_t^+ \rho_t^- \abs{u_t^+ - u_t^-}^2 - (p(\rho_t^+) - p(\rho_t^-))(\rho_t^+ - \rho_t^-) \right] = 0.   \label{eq: equivalent formula bis}
\end{equation} 
\end{proposition}

\begin{proof}
We have 
\begin{align}
    \div(\rho_t u_t) \otimes dt  = \div(\rho u) & = - \mu_t \otimes dt, 
    \\ \div(\rho_t u_t \otimes u_t + p(\rho_t) \textrm{Id} ) \otimes dt = \div( \rho u \otimes u + p(\rho) \textrm{Id} ) & = - \nu_t \otimes dt + f_t \otimes dt.  
\end{align}
By Fubini's theorem, recall that $f_t \in L^1_x$ for a.e. $t$. Since the disintegration is uniquely determined for a.e. $t$, we find a set $\mathcal{N} \subset (0,T)$ such that $\mathcal{L}^1(\mathcal{N}) =0$ and for any $t\in \mathcal{N}^c $ we have 
\begin{equation} \label{eq: space div is L^1}
    \div(\rho_t u_t) = -\mu_t \qquad \text{ and } \qquad \div(\rho_t u_t \otimes u_t + p(\rho_t) \textrm{Id} ) = -\nu_t + f_t.
\end{equation}
From now on, we assume that $t \in \mathcal{N}^c$ and fix a Lipschitz hypersurface $\Sigma \subset \Omega$ oriented by a unit vector $n \in \mathbb{S}^{d-1}$ such that both $\rho_t, u_t$ possess bilateral traces at $\Sigma$. By \eqref{eq: disintegrating momentum and continuity}, \eqref{eq: space div is L^1} and using \cref{P:div and normal trace}, \cref{t: distributional trace vs lebesgue trace} to compute the traces, we have\footnote{We point out that under our assumptions, $\div (\rho_t, u_t), \div(u_t \otimes u_t + \rho_t \textrm{Id})$ are only distributions in space. The fact that they are represented by Radon measures follows by the uniqueness of the disintegration.} 
\begin{align}
    \rho_t^+ (u_t^+)_n & = \rho_t^- (u_t^-)_n
    \\ \rho_t^+ (u_t^+)_n \,  (u_t^+ - u_t^-) & = - (p (\rho_t^+) - p(\rho_t^-)) \, n \label{eq: jump only normal}
\end{align}
Assume that $\rho_t^+ (u_t^+)_n =0$. In this case, we have $p(\rho^+_t) - p(\rho^-_t) = 0$. If $\rho_t^+ = 0$ or $\rho_t^- = 0$, then \eqref{eq: equivalent formula bis} holds. Thus, assume that $\rho^+_t \neq 0$ and $\rho^-_t \neq 0$. In this case, we have $(u_t^+)_n = (u_t^-)_n = 0$, and therefore \eqref{eq: equivalent formula bis} is valid. Assume that $\rho^+_t (u^+_t)_n \neq 0$. By \eqref{eq: jump only normal} the jump of $u_t$ is only along the normal component, so we have
$$ \abs{u_t^+ - u_t^-}^2 = \abs{(u_t^+)_n - (u_t^-)_n}^2. $$
Then we compute
\begin{align}
    \rho_t^+ \rho_t^- & \abs{u_t^+ - u_t^-}^2 - (p(\rho_t^+) - p(\rho_t^-))(\rho_t^+ - \rho_t^-) 
    \\ & = \rho_t^+ \rho_t^- \left( (u_t^+)_n - (u_t^-)_n)^2 + \rho_t^+ (\rho_t^+ - \rho_t^-)(u_t^+)_n \, ((u_t^+)_n - (u_t^-)_n\right) = 0. 
\end{align}
\end{proof}

\begin{corollary} \label{C: equivalence of formulas}
Under the assumptions of \cref{T: main bounded var}, suppose in addition that $\partial_t \rho, \partial_t(\rho u) \in \mathcal{M}_{x,t}$ have a disintegration with the properties of \cref{P: equivalent formula}. Then, the dissipation $\D$ satisfies  
\begin{equation} \label{eq: formula dissipation simplified}
    \D = - [\Pi(\rho_t^+, \rho_t^-) \, \div u_t \llcorner \J_{u_t}] \otimes dt = [ ((u^+_t)_n - (u^-_t)_n) \Pi (\rho_t^+, \rho_t^-) \mathcal{H}^{d-1} \llcorner \J_{u_t} ] \otimes dt,   
\end{equation}
where $\rho_t^\pm, u_t^\pm$ are the traces of $\rho_t, u_t$ in $\J_{u_t}$ that is oriented by a unit normal vector $n: \J_{u_t} \to \mathbb{S}^{d-1}$, respectively, $(u_t^\pm)_n = u^\pm \cdot n $ and $\Pi$ is defined by \eqref{eq: function Pi}.  
\end{corollary}

\begin{proof}
Let $\mathcal{N} \subset (0,T)$ be the negligible set given by \cref{P: equivalent formula} and fix $t \in \mathcal{N}^c$. Since $\J_{u_t}$ is $\mathcal{H}^{d-1}$-countably rectifiable, we infer that \eqref{eq: equivalent formula bis} holds in $\mathcal{H}^{d-1}$-a.e. point at $\J_{u_t}$. Thus, \eqref{eq: formula dissipation explicit} reduces to \eqref{eq: formula dissipation simplified}.  
\end{proof}

\subsection{About the general statement}

We conclude the discussion on the compressible isentropic Euler system by combining \cref{T: dissipation on hypersurfaces} and \cref{T: main bounded var}.

\begin{proof} [Proof of \cref{T: full statement}]
Since $\rho= 0$ almost everywhere in the interior of $A$, then $\D$ vanishes as a distribution on the interior of $A$. Without loss of generality, we may assume that $\rho^+ =0$ at $\mathcal{H}^d$-a.e. at $\partial A$. By \eqref{eq: momentum traces}, we have  
$$ (p(\rho^+) - p(\rho^-)) \, n_x =0 \qquad \mathcal{H}^{d}-\text{a.e. at }\partial A, $$
where $n= (n_t, n_x) \in \R \times \R^d$ is the unit norm vector defining the orientation of $\partial A$. If $ p(\rho^-) = p(\rho^+) =0 $, we infer that $\Pi(\rho^+, \rho^-) =0$. On the other hand, if $n_x =0$, we have $ u^+\cdot n_x = u^-\cdot n_x = 0 $. Hence, by \eqref{T: dissipation on hypersurfaces}, we have that $\abs{\D} (\partial A) =0$. Thus, by \cref{P:div and normal trace}, we we infer that $\abs{\D} (\overline{A}) =0$. To compute $\D \llcorner (I_n \times B_n)$, it suffice to apply \cref{T: main bounded var}. 
\end{proof}

\section{The inhomogeneous incompressible case} \label{s: IIE}

In the final section of this note, we discuss the inhomogeneous incompressible setting
\begin{equation} \label{eq: IIE}
    \begin{cases} \tag{IIE}
        \partial_t(\rho u) + \div(\rho u \otimes u) + \nabla p = \rho f, 
        \\ \partial_t \rho + \div(\rho u) = 0, 
        \\ \div(u) =0. 
    \end{cases}
\end{equation}
As in the isentropic compressible case, we consider weak solutions $u \in L^3_{t,x} , \rho \in L^\infty_{t,x}$. For the self-induced pressure, we require $p\in L^{\sfrac{3}{2}}_{t,x}$, and for the external force $f \in L^1_{t,x}$. The defect distribution is defined by 
\begin{equation} \label{eq: exact balance IIE}
    - \D : = \partial_t \left( \frac{\rho \abs{u}^2}{2} \right) + \div \left( u \left( \frac{\rho \abs{u}^2}{2} + p \right) \right) - \rho f \cdot u, 
\end{equation}
the latter being zero for smooth solutions. As in the compressible setting, the problem of finding a minimal set of assumptions to show that $\D \equiv 0$ can be approached in multiple ways. It is certainly possible to trade spatial regularity of the density for spatial smoothness of the velocity, as well as time regularity with no-vacuum on the density. See, for example, the analysis in \cite{FGSW17}. As explained in the introduction, here the main interest is to understand whether the formation of shocks may produce a non-trivial dissipation. To this end, we work in a functional setting where shocks naturally appear. This concludes the analysis started by the authors in \cite{InvVi24}. See \cref{s: tools} for the notation. To begin, we recall the following proposition, which is the analogue of \cref{T: dissipation on hypersurfaces}. 

\begin{proposition} [\cite{InvVi24}*{Theorem 1.3}] \label{P: no diss IIE}
Let $(\rho, u, p) \in L^\infty_{x,t}$ be a weak solution to \eqref{eq: IIE} with external force $f \in L^1_{t,x}$ and dissipation $\D \in \mathcal{M}_{t,x}$. Let $\Sigma \subset (0,T) \times \Omega$ be a $\mathcal{H}^d$-countably rectifiable set and assume that $(\rho, u, p)$ have bilateral traces at $\Sigma$. Then $\D \llcorner \Sigma \equiv 0$. 
\end{proposition}

Then, we assume no vacuum and establish the analogue of \cref{t: main DR}.

\begin{proposition}  \label{T: main IIE}
Let $(\rho, u, p)$ be a weak solution to \eqref{eq: IIE} with external force $f$. Assume that $u\in L^3_{t,x}, p\in L^{\sfrac{3}{2}}_{t,x}, \rho \in L^\infty_{t,x}$ satisfies \eqref{no vacuum}, and $f \in L^{\sfrac{3}{2}}_{t,x}$\footnote{Any assumption that makes $p\cdot u, f \cdot u\in L^1_{t,x}$ would suffice.}. The distribution $\D$ defined by \eqref{eq: exact balance IIE} satisfies 
\begin{equation} \label{eq: defect distribution IIE}
    \D = \lim_{\ell \to 0} \left( J_1^\ell - J_3^\ell \right) , 
\end{equation}
where the limit is taken in the sense of distributions, $J_1^\ell$ is defined by \eqref{eq: J_1} and $J_3^\ell$ satisfies 
\begin{equation} \label{eq: J_3}
    J_3^\ell := \overline{p}  \div(\widetilde{u}). 
\end{equation} 
\end{proposition}

\begin{proof}
The proof follows line by line the argument outlined in \textsc{Step 1} of the proof of \cref{t: main DR}. 
\end{proof}

The problem of finding an approximation formula for the defect distribution $\D$ has also been addressed in \cite{InvVi24}*{Theorem 1.2}. In that case, the approach relies on testing the equation with $\widetilde{u}$, as done by Duchon and Robert \cite{DR00} in the homogeneous incompressible setting. Here, instead, we test the mollified equation with $\widetilde{u}$. This is similar to the Constantin--E--Titi approach \cite{CET94} in the homogeneous incompressible setting. This gives rise to different error terms, compared to \cite{InvVi24}, which can be explicitly computed by the blow up analysis in the bounded variation/deformation setting. The following is the analogue of \cref{T: main bounded var}.

\begin{theorem}
Let $(\rho, u, p)$ be a weak solution to \eqref{eq: IIE}. Assume that $u \in L^\infty_{t,x} \cap L^1_t BD_x, \rho \in L^\infty_{t,x} \cap L^1_t BV_x$ satisfies \eqref{no vacuum} and $p, f \in L^1_{t,x}$. We have 
$$ -\D = \lim_{\ell \to 0} \overline{p} \div(\widetilde{u}) $$
in the sense of distributions. Suppose in addition that $p \in L^\infty_{t,x}$ and that 
\begin{equation} \label{eq: condition on the pressure} 
    \left( \abs{\nabla \rho_t} + \abs{E u_t} \right) (\mathcal{B}_{p_t}^c) = 0 \qquad \text{ for a.e. $t$}, 
\end{equation}
where $\mathcal{B}_{p_t}$ is defined in \cref{ss:jumps}. Then $\D \equiv 0$. 
\end{theorem}

Some remarks are in order. 

\begin{itemize}
    \item[(i)] The condition \eqref{eq: condition on the pressure} on the pressure can be avoided in the homogeneous incompressible case, as shown in \cites{Inv25, DRINV23}. In that case, the equation is tested with $\overline{u}$, the term $J_1^\ell$ defined by \eqref{eq: J_1} coincides with that given by the Constantin--E--Titi approximation \cite{CET94} and $\overline{p} \div(\overline{u}) =0 $ pointwise for any $\ell >0$. However, in the inhomogeneous setting, an additional condition on the pressure such as \eqref{eq: condition on the pressure} seems to be necessary, since we consider $\widetilde{u}$ as a test function, and it is expected that the term $\overline{p} \div(\widetilde{u})$ vanishes only in the limit. We note that \eqref{eq: condition on the pressure} is satisfied if we assume that $p_t \in BV_x$ for a.e. $t$. In the stationary case, we observe that 
$$\nabla p = -\div(\rho u \otimes u) + \rho f, $$
that is a measure if we assume both $\rho, u \in BV_x$ (this follows by a modification of \cref{L: product rule BV BD}, see also \cite{AFP00}*{Chapter 3}).  We also point out that the $L^\infty$ bound on the pressure does not seem to follow by those on $\rho, u$ and, hence, we need to assume it. 
\item[(ii)] The reader might compare our \cref{T: main IIE} with \cite{InvVi24}*{Theorem 1.2}. In that case, an upper bound on the dissipation is given in terms of the singular part of the gradient of the density and the symmetric gradient of the velocity. Here, at the price of the assumption \eqref{eq: condition on the pressure}, we show that the dissipation vanishes for solutions in the same functional setting. 
\end{itemize}

The proof of \cref{T: main IIE} follows that of \cref{T: main bounded var}, with the help of \cref{P: J^1} and the following proposition. 

\begin{proposition} \label{P: J_3}
Let $\rho \in L^\infty \cap BV$ and let $u \in L^\infty \cap BD$ be divergence free. Assume that $p \in L^\infty$ and $(\abs{\nabla \rho} + \abs{E u}) (\mathcal{B}_p^c) = 0$. Fix a kernel in $\mathcal{K}_{\rm rad}$ and define $J_3^\ell$ by \eqref{eq: J_3}. Then, we have
\begin{equation}
    \lim_{\ell \to 0} J_3^\ell = 0
\end{equation}
weakly in the sense of measures. 
\end{proposition}

\begin{proof}
We observe that 
\begin{equation}
    J_3^\ell = \frac{\overline{p}}{\overline{\rho}} \div(\overline{\rho u}) - \frac{\overline{p} }{(\overline{\rho})^2} \overline{\rho u} \cdot \nabla \overline{\rho} =: J_{3,1}^\ell - J_{3,2}^\ell. 
\end{equation}
We study both terms separately. 

\textsc{\underline{The term $J_{3,1}^\ell$}:} We define 
$$F_{3,1}(c,d) := \frac{d}{c} \qquad (c,d) \in \R \times \R. $$
We set $W:= (\rho, p(\rho))$. We combine \cref{L: product rule BV BD}, \cref{P: prop BV}, \cref{P: prop BD}, \cref{R: blow up of composition} as in the proof of \cref{P: J^1}. Then, by \cref{P: mollification of function and measure}, we infer that 
\begin{equation}
    \lim_{\ell \to 0} J_{3,1}^\ell = \left( \int_0^1 F_{3,1} (s W^+ + (1-s) W^-)\, ds \right) \div(\rho u) =  \left( \int_0^1 F_{3,1} (s W^+ + (1-s) W^-)\, ds \right) \frac{u^+ + u^-}{2} \cdot \nabla \rho, 
\end{equation} 
where the limit is taken weakly in the sense of measures. 

\textsc{\underline{The term $J_{3,2}^\ell$}:} We define 
$$F_{3,2} (b,c,d) := \frac{d b}{c^2} = F_{3,1} (c,d) \, \frac{b}{c} \qquad (b,c,d) \in \R^d \times \R \times \R. $$
We combine \cref{R: negligible set of BD BV}, \cref{P: prop BV}, \cref{P: prop BD}, \cref{R: blow up of composition} as in the proof of \cref{P: J^1}. We set $Z:= (\rho u, \rho, p(\rho)) = (\rho u, W)$ and, by \eqref{eq: condition on the pressure} and \cref{P: mollification of function and measure}, we find
\begin{align}
    \lim_{\ell \to 0} J_{3,2}^\ell = \left( \int_0^1 F_{3,1}(s W^+ + (1-s) W^-) \, \frac{s \rho^+ u^+ + (1-s) \rho^- u^-}{ s \rho^+ + (1-s) \rho^-} \, ds \right) \cdot \nabla \rho, 
\end{align}
where the limit is taken in the sense of measures. 

\textsc{\underline{Summing all terms}:} In summary, we have 
\begin{align}
    \lim_{\ell \to 0} J_{3,1}^\ell - J_{3,2}^\ell & = \left( \int_0^1 F_{3,1}(sW^+ + (1-s) W^-) \left( \frac{u^+ + u^-}{2} - \frac{s \rho^+ u^+ + (1-s) \rho^- u^-}{ s \rho^+ + (1-s) \rho^-} \right) \, ds  \right) \cdot \nabla \rho
    \\ & = \left( \frac{u^+ - u^- }{2} \int_0^1 F_{3,1} (sW^+ + (1-s) W^-) \, ds \right) \cdot \nabla \rho.  
\end{align}
Since the latter vanishes on $\S_u^c$, we restrict it to $\J_u$. Thus, by \cref{P: prop BV} we have 
\begin{equation}
    \lim_{\ell \to 0} J_{3,1}^\ell - J_{3,2}^\ell = \left( \frac{(u^+ - u^-)\cdot n }{2} \int_0^1 F_{3,1} (sW^+ + (1-s) W^-) \, ds \right) (\rho^+ -\rho^-) \, \mathcal{H}^{d-1} \llcorner \J_u = 0,
\end{equation}
since $u$ is divergence free.
\end{proof}

Finally, we discuss the proof of \cref{T: main IIE}.

\begin{proof} [Proof of \cref{T: main IIE}]
The proof of the first statement follows by \cref{P: J^1} and the divergence free condition. In particular, as showed in the proof of \cref{T: main bounded var}, it turns out that $J_1^\ell \rightharpoonup 0$ weakly as measures. Assume now that \eqref{eq: condition on the pressure} is satisfied. Fix a test function $\phi \in C_c(\Omega \times (0,T))$. By \cref{P: J_3} and since the sequence $t\mapsto \langle  J_3^\ell(t) , \phi_t\rangle $ is bounded in $L^\infty_t$, we use the dominated convergence theorem with respect to the time variable to conclude  
\begin{equation}
    \lim_{\ell \to 0} \langle J_{3}^\ell, \phi \rangle = \int_0^T \lim_{\ell \to 0} \langle J_3^\ell(t) , \phi_t \rangle \, dt = 0. 
\end{equation}
\end{proof}

As in the isentropic compressible setting, we propose a general statement that collects \cref{P: no diss IIE} and \cref{T: main IIE}. See \cref{s: tools} for the notation. 

\begin{corollary} \label{T: full statement IIE} 
Let $(\rho, u, p) \in L^\infty_{t,x}$ be a weak solution to \eqref{eq: IIE} with external force $f \in L^1_{t,x}$ and dissipation $\D$ as in \eqref{eq: defect distribution IIE} given by a space-time Radon measure. Suppose that $\Omega \times (0,T) = A \cup B$, where $A, B$ satisfy the following conditions:  
\begin{enumerate}
    \item $A$ is a closed set and $\rho (t,x) =0$ $\mathcal{L}^{d+1}$-almost everywhere in the interior of $A$; 
    \item $\partial A$ is $\mathcal{H}^d$-countably rectifiable and $\rho,u, p$ possess bilateral space-time traces at $\partial A$; 
    \item there are intervals $I_n \subset (0,T)$ and open sets $B_n \subset \Omega$ such that 
    $$B = \bigcup_{n} I_n \times B_n, $$
    $\rho \in L^1(I_n, BV(B_n)), u \in L^1(I_n, BD(B_n))$ and there are constants $c_n>0$ such that 
    \begin{equation} \label{eq: local no vacuum}
        \rho(t,x) \geq c_n >0 \qquad \text{a.e. at } I_n \times B_n; 
    \end{equation} 
    \item for any $n \in \N$ we have
\begin{equation} \label{eq: condition on the pressure bis} 
    \left( \abs{\nabla \rho_t} + \abs{E u_t} \right) (\mathcal{B}_{p_t}^c \cap B_n) = 0 \qquad \text{ for a.e. $t \in I_n$}. 
\end{equation}
\end{enumerate} 
Then $\D \equiv 0$.  
\end{corollary}

\begin{proof}
Since $\rho$ vanishes almost everywhere in the interior of $A$, by the momentum equation, we have that $\nabla p =0$ as a distribution in the interior of $A$. Since $\div(u) = 0$, it follows that $\div(u p ) =0$ as a distribution in the interior of $A$. By \cref{P: no diss IIE} and \cref{P:div and normal trace}, we have $\abs{\D} (\overline{A}) =0$. On the other hand, by \cref{T: main IIE}, we have that $\D \llcorner (I_n \times B_n) \equiv 0$ for any $n \in \N$. Therefore we conclude that $\D\llcorner B =0$. 
\end{proof}

\bibliographystyle{plain} 
\bibliography{biblio}

\end{document}